\documentclass[final,3p,times]{elsarticle}

\usepackage{algorithm}
\usepackage{algorithmic}
\usepackage{multirow}
\usepackage{xcolor}
\usepackage{hyperref}
\usepackage{url}
\usepackage{booktabs,threeparttable,caption2}
\usepackage{amsthm,amsmath}
\usepackage{float}
\usepackage{mathrsfs}
\usepackage{latexsym}
\usepackage{tabularx}
\usepackage{diagbox}
\biboptions{numbers,sort&compress}
\newtheorem{thm}{Theorem}

\newtheorem{lem}{Lemma}

\newdefinition{example}{Experiment}
\newdefinition{rmk}{Remark}

\newcaptionstyle{left}{
\usecaptionmargin\captionfont
{\flushleft\bfseries\captionlabelfont\captionlabel\par}
\mbox{\onelinecaption{\captiontext}{\captiontext}}
}

\begin{document}

\begin{frontmatter}


\title{Computing the Lyapunov operator $\varphi$-functions, with an application to matrix-valued exponential integrators}
\author[author1,author2]{Dongping Li }
\ead{lidp@ccsfu.edu.cn}
\author[author2]{Yue Zhang}
\author[author3]{Xiuying Zhang \corref{cor1}}
\ead{xyzhang009@gmail.com}
\cortext[cor1]{Corresponding author.}
\address[author1]{School of Mathematics, Jilin University, Changchun 130012, PR China}
\address[author2]{ School of Mathematics, Changchun Normal University, Changchun 130032, PR China}
\address[author3]{International Education Teachers School, Changchun Normal University, Changchun 130032, PR China}

\begin{abstract}

In this paper, we develop efficient and accurate evaluation for the Lyapunov operator function
$\varphi_l(\mathcal{L}_A)[Q],$ where $\varphi_l(\cdot)$ is the function related to the exponential,
$\mathcal{L}_A$ is a Lyapunov operator and $Q$ is a symmetric and full-rank matrix.
An important application of the algorithm is to the matrix-valued exponential
integrators for matrix differential equations such as differential Lyapunov equations
and differential Riccati equations.  The method is exploited by using the modified scaling and squaring
procedure combined with the truncated Taylor series. A quasi-backward error analysis is presented to
determine the value of the scaling parameter and the degree of the Taylor approximation.
Numerical experiments show that the algorithm performs well in both accuracy and efficiency.
\end{abstract}

\end{frontmatter}
\section{Introduction}
In this work we are concerned with numerical method for approximating the so-called Lyapunov operator $\varphi$-functions of the form
\begin{equation}\label{1.1}
\varphi_l(\mathcal{L}_A)[Q], ~~l\in \mathbb{N}.
\end{equation}
where $Q\in \mathbb{R}^{N\times N}$ is symmetric and full rank, and
$\mathcal{L}_A:\mathbb{R}^{N\times N}\rightarrow \mathbb{R}^{N\times N}$ is the \emph{Lyapunov~operator}
\begin{equation}\label{1.2}
\begin{aligned}
\mathcal{L}_A[X] = AX+XA^T, ~~ A\in \mathbb{R}^{N\times N}.
\end{aligned}
\end{equation}
These $\varphi$-functions are defined for integers $l\geq 0$  by the functional integral
\begin{equation}\label{1.3}
\begin{aligned}
\varphi_{0}(\mathcal{L}_A)=e^{\mathcal{L}_A},~~\varphi_{l}(\mathcal{L}_A)=\frac{1}{(l-1)!}\int_{0}^{1}e^{(1-\theta)\mathcal{L}_A}\theta^{l-1}d\theta, ~l\ge 1.
\end{aligned}
\end{equation}

Let ${\mathcal{L}_A}^k$ denote the $k$th power of the Lyapunov operator $\mathcal{L}_A$, defined as $k$-fold composition,
i.e., ${\mathcal{L}_A}^k\equiv \mathcal{L}_A[{\mathcal{L}_A}^{k-1}]$ for $k>0$, and ${\mathcal{L}_A}^0=I_N.$
The Lyapunov operator $\varphi$-functions (\ref{1.1}) can then be represented as the Taylor series expansion
\begin{equation}\label{1.3a}
\varphi_l(\mathcal{L}_A)=\sum\limits^{\infty}_{k=0}\frac{1}{(k+l)!}{\mathcal{L}_A}^k, ~~l\in \mathbb{N},
\end{equation}
which satisfy the recursive relation
\begin{equation}\label{1.3b}
\varphi_{k-1}(\mathcal{L}_A)=\mathcal{L}_A\varphi_k(\mathcal{L}_A)+\frac{1}{(k-1)!}I_N,~~k=l, l-1,\ldots,1.
\end{equation}
Furthermore, we have
\begin{equation}\label{1.3c}
\varphi_l(\mathcal{L}_A)=\mathcal{L}_A^{-l} \left(e^{\mathcal{L}_A}-\sum\limits^{l-1}_{j=0}\mathcal{L}_A^{j}/j!\right).
\end{equation}

Such problems play a key role in a class of numerical methods called matrix-valued exponential integrators for solving matrix differential equations (MDEs) of the form
\begin{equation}\label{1.4}
\left\{
\begin{array}{l}
X'(t)=AX(t)+X(t)A^T+N(t,X(t)),~~\\
X(t_0)=X_0,
\end{array}
\right.
\end{equation}
where $N:\mathbb{R}\times \mathbb{R}^{N\times N}\rightarrow \mathbb{R}^{N\times N}$ is the nonlinear term, and $X(t)\in \mathbb{R}^{N\times N}$.
MDEs are of major importance in many fields such as optimal control, model reduction of linear dynamical systems,
semi-discretization of a partial differential equation and many others (see e.g., \cite{Abou,Antoulas,Ascher,Jacobs}).
Many important equations such as differential Lyapunov equations (DLEs) and differential Riccati equations (DREs) can be put in the form.
In the literature, there has been an enormous approaches to compute the solution of MDEs (\ref{1.4}),
see, e.g., \cite{Behr,Benner01,Choi,Dieci,Koskela,Hached,Mena,Simoncini20,Stillfjord2}.

Exponential integrators constitute an interesting class of numerical methods for the time integration of stiff systems of differential equations.
The methods are very competitive for semi-linear stiff problems as they can treat the linear term exactly and the nonlinearity in an explicit way.
For the standard (vector-valued) exponential integrators, we refer to \cite{BV2005, Hochbruck2010} for a full review.
Although MDEs (\ref{1.4}) can be reformulated as a standard (vector-valued) ordinary differential equation and solved by
a standard exponential integrator, this approach will be usually memory consuming as well as computationally
expensive. Recently, in \cite{Li2021}, the matrix-valued exponential Rosenbrock-type methods are proposed for solving DREs.

The important ingredient to implementation of vector-valued exponential integrators is the computation of
the matrix $\varphi$-functions. But for the matrix-valued exponential integrators,
a few operator $\varphi$-functions are required to compute at each time step.
For matrix $\varphi$-functions many numerical methods have been studied, see e.g.,
\cite{AlMohy2011,Tokman18,Li2022,Niesen2012,Sidje1998, Skaflestad,Suhov}.
For the operator $\varphi$-functions and to our knowledge there is no existing method in the literature.

The scaling and squaring method is the most popular
method for computing the matrix exponential \cite{Moler2003}. In \cite{Skaflestad},
a modified scaling and squaring method based on Pad\'e approximation is described for the computation of matrix $\varphi$-functions.
A very recent paper \cite{Li2022} shows the modified scaling and squaring procedure combined with truncated Taylor series
could be more efficient. The aim of the present paper is to generalize the techniques as those used in \cite{Li2022}
to accurately and efficiently evaluate the operator $\varphi$-functions of the form (\ref{1.1}).
We present a quasi-backward error analysis to help choosing the key parameters of the method.
Numerical experiments illustrate that the method can be used as a kernel of matrix-valued exponential integrators.

The paper is organized as follows. In Section 2, the matrix-valued exponential integrators
 are introduced for the application to MDEs.
 In Section 3, we introduce the modified scaling and squaring method for
evaluating the operator $\varphi$-functions. The implementation details of the method are presented in Section 4.
Numerical experiments are given to illustrate the benefits of the algorithm in Section 5. Finally, conclusions are given in Section 6.

Throughout the paper, we use the following notations.

$\bullet$ $I_n$ is the $n\times n$ identity matrix, and $0_n$ is the $n\times n$ zero matrix.

$\bullet$ For a matrix $A=[a_{ij}]\in \mathbb{R}^{n\times n},$ its column-wise vector is denoted by
 \begin{equation*}
\text {vec}(A)=[a_{11}, a_{21},\ldots, a_{n1}, a_{12}, a_{22}, \ldots, a_{n2}, a_{1n}, a_{2n}, \ldots, a_{nn}]^T.
\end{equation*}

$\bullet$ $\|\cdot\|$ denotes any consistent matrix norm. In particular $\|\cdot\|_1$ denotes
the 1-norm of matrices and $\|\cdot\|_F$ denotes the Frobenius norm of matrices.

$\bullet$ $\sigma(\cdot)$ and $\rho(\cdot)$ denote the spectral set and the spectral radius of matrix or operator, respectively.

$\bullet$ $A\otimes B=[a_{ij}B]\in \mathbb{R}^{mn\times mn}$ and $A\oplus B =A\otimes I_n+I_m\otimes B$, respectively,
denote the Kronector product and the Kronector sum
of matrices $A=[a_{ij}]\in \mathbb{R}^{m\times m}$ and $B\in \mathbb{R}^{n\times n}.$

$\bullet$ $\text{Lyap}(n)$ denotes the set of Lyapunov operator $\mathcal{L}_A[X]=AX+XA^T$ for any $A\in \mathbb{R}^{n\times n}.$

$\bullet$ $\lfloor x\rfloor$ denotes the largest integer not exceeding $x$ and $\lceil x\rceil$ denotes the smallest integer not less than $x$.

In this paper, the norm of Lyapunov operator $\mathcal{L}_A\in$\text{Lyap}($n$) is the induced norm, which is defined by

\begin{equation}\label{1.5}
\|\mathcal{L}_A\|:=\max\limits_{X\in \mathbb{R}^{n\times n}}\frac{\|\mathcal{L}_A[X]\|}{\|X\|}.
\end{equation}
From (\ref{1.5}) it is easy to show that
\begin{equation}\label{1.6}
\rho(\mathcal{L}_A)\leq\|\mathcal{L}_A\|\leq 2\|A\|.
\end{equation}

\section{Matrix-valued exponential integrators}\label{sec:2}
The matrix-valued exponential integrators for (\ref{1.4}) can be derived from the approximation of the integral that results from the application of the variation-of-constants
formula. By means of the variation-of-constants formula (see e.g., \cite{Kucera}), the exact solution of (\ref{1.4}) at time $t_{n+1}=t_n+h_n$
satisfies the nonlinear integral equation
\begin{equation}\label{2.1}
\begin{aligned}
X(t_n+h)=e^{hA}X(t_n)e^{hA^T}+h\int_{0}^{1}e^{(1-\tau)hA}N(t_n+\tau h,X(t_n+\tau h))e^{(1-\tau)hA^T}\text{d}\tau.
\end{aligned}
\end{equation}
where $h$ is the time step. The following lemma provides a more compact form for the solution formula (\ref{2.1}).
\begin{lem} [\cite{Behr}]
For the Lyapunov operator $\mathcal{L}_A:\mathbb{R}^{N\times N}\rightarrow \mathbb{R}^{N\times N}$ and its partial realizations $\mathcal{H}, \mathcal{H}^T: \mathbb{R}^{N\times N}\rightarrow \mathbb{R}^{N\times N},$ $\mathcal{H}[X]=AX, \mathcal{H}^T=XA^T,$ it holds that:
\begin{equation}\label{2.1a}
e^{h\mathcal{L}_A}[X]=e^{h\mathcal{H}}e^{h\mathcal{H}^T}[X]=e^{hA}Xe^{hA^T}.
\end{equation}
\end{lem}

Thus, the solution formula (\ref{2.1}) can be rewritten as
\begin{equation}\label{2.2}
\begin{aligned}
X(t_n+h)=e^{h\mathcal{L}_A}[X({t_n)}]+h\int_{0}^{1}e^{(1-\tau)h\mathcal{L}_A}[N(t_n+\tau h,X(t_n+\tau h))]\text{d}\tau.
\end{aligned}
\end{equation}
By approximating the nonlinear terms
$N\big(t_n+sh, X(t_n+s)\big)$ in (\ref{2.2}) by an appropriate interpolating polynomial, we can exploit various types of
exponential integrators like Runge-Kutta and multi-step methods. For example, by interpolating the
nonlinearity at the known value $N\big(X(t_n)\big)$ only, we obtain the well-known exponential Euler scheme:
\begin{eqnarray}\label{2.3}
X_{n+1}=e^{h\mathcal{L}_A}[X_n]+h\varphi_{1}(h\mathcal{L}_A)\left[N(t_n,X_n)\right].
\end{eqnarray}
The scheme (\ref{2.3}) is first order and accurate for MDEs (\ref{1.4}) with $N(t,X)$ being a constant matrix.
The application of the standard exponential Runge-Kutta type methods \cite{MH2}, to the matrix-valued initial value
problem (\ref{1.4}), yields
\begin{equation}\label{2.4}
\left\{
\begin{array}{l}
X_{ni}=e^{c_ih_n\mathcal{L}_A}[X_n]+h_n\sum\limits^{i-1}_{j=1}a_{ij}(h_n\mathcal{L}_A)[N(t_n+c_j, X_{nj})],~~1\leq i\leq s,\\
X_{n+1}=e^{h_n\mathcal{L}_A}[X_n]+h_n\sum\limits^{s}_{i=1}b_{i}(h_n\mathcal{L}_A)[N(t_n+c_i, X_{ni})].
\end{array}
\right.
\end{equation}
Here, $c_i$ is the nodes, and the coefficients $a_{ij}(z),$ $b_i(z)$ are linear combinations of $\varphi_j(c_iz),$ $\varphi_j(z),$ respectively.
Details on the values of these coefficients and convergence analysis are exactly similar with the standard exponential integrators and can be found in \cite{MH2, Luan2013}.

In particular, for the DLEs
\begin{equation}\label{2.5}
\left\{
\begin{array}{l}
X'(t)=AX(t)+X(t)A^T+\frac{t^{l-1}}{(l-1)!}Q,~~l\in \mathbb{N},\\
X(0)=0_{N},
\end{array}
\right.
\end{equation}
the exact solutions at time $t$ are given as $X(t)=t^l\varphi_l(t\mathcal{L}_A)[Q]$.
By reformulating the LDEs (\ref{2.5}) as the vector-valued ordinary differential equations, one can easily show
that
\begin{equation}\label{2.6}
\text{vec}\left(\varphi_l(t\mathcal{L}_A)[Q]\right)=\varphi_l(tL_A)b,
\end{equation}
where $L_A=A\oplus A \in \mathbb{R}^{N^2\times N^2}$ and $b=\text{vec}(Q)\in \mathbb{R}^{N^2}.$
Thus it is then possible to directly apply a method tailored for the matrix $\varphi$-functions to evaluate these operator $\varphi$-functions.
However, this approach would lead to large memory and computational requirements.

As can be observed, the Lyapunov operator $\varphi$-functions appear naturally in the matrix-valued exponential
integrators. The efficient and accurate evaluation of these functions is crucial for stability and speed of exponential integrators.

\section{The method}\label{sec:3}
This section we briefly introduce the modified scaling and squaring for evaluating the operator $\varphi$-functions.
The following lemma gives a formula for the operator $\varphi$-functions. The formula for scalar arguments has been considered already in \cite{Skaflestad} without proof.
\begin{lem} Given a Lyapunov operator $\mathcal{L}_A:\mathbb{R}^{N\times N}\rightarrow \mathbb{R}^{N\times N}$ and an integer $l\geq0,$ then for any $a, b \in \mathbb{R},$ we have
\begin{equation}\label{2.7}
\varphi_l\left((a+b)\mathcal{L}_A\right)=\frac{1}{(a+b)^l}\left(a^l\varphi_0(b\mathcal{L}_A)\varphi_l(a\mathcal{L}_A)+\sum\limits^{l}_{k=1}\frac{a^{l-k}b^{k}}{(l-k)!}\varphi_{k}(b\mathcal{L}_A)\right).
\end{equation}
\end{lem}
\begin{proof} For any $Q\in \mathbb{R}^{N\times N},$ it is clear that the solution of DLEs (\ref{2.5}) at time $a+b$ is
\begin{equation}\label{2.8}
(a+b)^l\varphi_l\left((a+b)\mathcal{L}_A\right)[Q].
\end{equation}
On the other hand, by splitting the time interval $[0, a+b]$ into two subintervals $[0, a]$ and $[a, a+b],$ we can re-express the solution of DLEs (\ref{2.5}) at time $a+b$ by using a time-stepping method. At time $a,$ the solution is $X(a)=a^l\varphi_l\left(a\mathcal{L}_A\right)[Q].$
To advance the solution, using $X(a)$ as initial value and applying the formula (\ref{2.2}), we arrive at
\begin{eqnarray}\label{2.9}
\begin{aligned}
X(a+b)&=e^{b\mathcal{L}_A}[X(a)]+b\int_{0}^{1}e^{(1-\tau)b\mathcal{L}_A}\left[\frac{(a+\tau b)^{l-1}}{(l-1)!}Q\right]\text{d}\tau\\
&=a^le^{b\mathcal{L}_A}\varphi_l(a\mathcal{L}_A)[Q]+b\int_{0}^{1}\sum\limits^{l-1}_{k=0}\frac{a^{l-k-1}(b\tau)^k}{(l-1-k)!k!}e^{(1-\tau)b\mathcal{L}_A}[Q]\text{d}\tau\\
&=a^le^{b\mathcal{L}_A}\varphi_l(a\mathcal{L}_A)[Q]+\sum\limits^{l}_{k=1}\frac{a^{l-k}b^k}{(l-k)!}\varphi_k(b\mathcal{L}_A)[Q].
\end{aligned}
\end{eqnarray}
By equalizing the expression (\ref{2.8}) with (\ref{2.9}), we have the claim directly.
\end{proof}

On taking $a=b=1,$ we obtain
\begin{equation}\label{2.10}
\varphi_l(2\mathcal{L}_A)[Q]=\frac{1}{2^l} \left(\varphi_0(\mathcal{L}_A)\varphi_l(\mathcal{L}_A)[Q]+ \sum\limits^{l}_{j=1}\frac{1}{(l-j)!}\varphi_j(\mathcal{L}_A)[Q]\right).
\end{equation}
The formula is the starting point for the method for the evaluation of (\ref{1.1}) which we develop in the next subsection.

\subsection{Derivation of the method}\label{sec:3.1}
 The main idea behind this method is to scale the Lyapunov operator $\mathcal{L}_{A}$ by a factor $2^{-s}$,
 so that $\|2^{-s}\mathcal{L}_A\|$ is sufficiently small and $\varphi_j(2^{-s}\mathcal{L}_A),$ $j=0,1,\ldots,l,$
 can be well approximated by their truncated Taylor series.
 For simplicity of exposition, in the following we will use $\mathcal{L}$ instead of $2^{-s}\mathcal{L}_{A}$. Then, we can compute $\varphi_l(\mathcal{L}_A)[Q]$ via the following coupled recursions
\begin{eqnarray}\label{3.11}
\varphi_i(2^{k}\mathcal{L})[Q]=\frac{1}{2^i} \left(\varphi_0(2^{k-1}\mathcal{L}) \varphi_i(2^{k-1}\mathcal{L})[Q]+ \sum\limits^{i}_{j=1}\frac{1}{(i-j)!}\varphi_j(2^{k-1}\mathcal{L})[Q]\right), ~1\leq i\leq l,~k=0,1,\dots,s.
\end{eqnarray}
This process need to pre-evaluate $\varphi_j(\mathcal{L})[Q]$ for $j=1,\dots,l.$ Let
\begin{eqnarray}\label{3.12}
T_{l,m}(\mathcal{L})[Q]=\sum\limits^{\infty}_{k=0}\frac{1}{(k+l)!}{\mathcal{L}}^k[Q]
\end{eqnarray}
be the order of $m$ truncated Taylor approximation to $\varphi_l(\mathcal{L}).$ The operator polynomial $T_{l,m}(\mathcal{L})[Q]$ can be computed by using the
Honer's method, which requires $m$ matrix-matrix products. Once $T_{l,m}(\mathcal{L})[Q]$ is computed, the other $\varphi_j(\mathcal{L})[Q]$ can be recursively evaluated by using the relation (\ref{1.3b}), that is
\begin{equation}\label{3.13}
\varphi_j(\mathcal{L})[Q]\approx T_{j,m}(\mathcal{L})[Q]:=\mathcal{L}T_{j+1,m}(\mathcal{L})[Q]+\frac{1}{j!}Q,~~j=l-1,l-2,\cdots,1.
\end{equation}
Obviously, $T_{k,m}(\mathcal{L})$ is the order of $m+l-k$ truncated Taylor approximation to $\varphi_k(\mathcal{L}),$ i.e., $\varphi_k(\mathcal{L})=T_{k,m}(\mathcal{L})+\mathcal{O}(\|\mathcal{L}^{m+l-k+1}\|)$. This process involves $l-1$ additional matrix-matrix products.

As the main ingredient we require a method to implement the operator exponential $\varphi_0(2^{k}(\mathcal{L})[\cdot]$ involved in (\ref{3.11}). Here we use "$\cdot$" to denote the matrix being acted on. An observation is based on the order $m+l$ truncated Taylor series $T_{0,m}(\mathcal{L})$ of $\varphi_0(\mathcal{L}),$ i.e.,
\begin{equation}\label{3.14}
 \varphi_0(\mathcal{L})\approx T_{0,m}(\mathcal{L}) := \sum\limits^{m+l}_{k=0}\frac{{\mathcal{L}}^k}{k!}.
\end{equation}
Since $\varphi_0(2^{k}\mathcal{L})=\varphi_0(\mathcal{L})^{2^{k}},$ we can naturally approximate $\varphi_0(2^{k}\mathcal{L})$  by $T_{0,m}({\mathcal{L}})^{2^k}.$
Substituting all the above approximations into (\ref{3.11}), we then recursively evaluate $\varphi_i(\mathcal{L}_A)[Q]$ as
\begin{eqnarray}\label{3.15}
T_{i,m}(2^k\mathcal{L})[Q]=\frac{1}{2^i} \left(\left(T_{0,m}(\mathcal{L})\right)^{2^{k-1}} \left[T_{i,m}(2^{k-1}\mathcal{L})[Q]\right]+ \sum\limits^{i}_{j=1}\frac{1}{(i-j)!}T_{j,m}(2^{k-1}\mathcal{L})[Q]\right),~1\leq i\leq l
\end{eqnarray}
for $k=0,1,\dots,s.$

\subsection{Choice of the parameters}\label{sec:3.2}
The above procedure has two key parameters, the scaling parameter $s$ and the degree $m$ of operator polynomial $T_{l,m}(\mathcal{L}).$ These need to be chosen
appropriately. We are using a quasi-backward error analysis to determine these parameters.

Let
\begin{equation}\label{3.21}
\Omega_m:=\{ \mathcal{L}:~~\rho (e^{-\mathcal{L}}T_{0,m}(\mathcal{L})-I_N)<1,~\mathcal{L}\in \text{Lyap}(N)\},
\end{equation}
where $\rho$ is the spectral radius and $T_{0,m}$ is defined as in (\ref{3.14}). Then the operator function
\begin{equation}\label{3.22}
h_{m+l}(\mathcal{L})=\log(e^{-\mathcal{L}}T_{0,m}(\mathcal{L}))
\end{equation}
is defined for $\mathcal{L} \in \Omega_m,$ and it commutes with $\mathcal{L}$, where $\log$ denotes the principal logarithm.
Over $\Omega_m,$  the function
$h_{m+l}(\mathcal{L})$ has an infinite power series expansion
\begin{equation}\label{3.23}
h_{m+l}(\mathcal{L})=\sum\limits^\infty_{k=m+l+1}c_k{\mathcal{L}}^{k}.
\end{equation}
The following theorem provides a quasi-backward error for the recursions (\ref{3.15}), which is a useful tool in choosing suitable parameters $m$ and $s.$
\begin{thm}\label{th1}Let $2^{-s}\mathcal{L}_A\in \Omega_m,$ the approximation $T_{l,m}(\mathcal{L}_A)$ generated by
recursions (\ref{3.15}) satisfies
\begin{eqnarray}\label{3.24}
T_{l,m}(\mathcal{L}_A)={\mathcal{L}_A}^{-l}\left(e^{\mathcal{L}_A+\Delta \mathcal{L}_A}-\sum\limits^{l-1}_{j=0}{\mathcal{L}_A}^j/j!\right),~~i=1,2,\ldots,l,
\end{eqnarray}
where
\begin{equation}
\Delta \mathcal{L}_A:=2^sh_{m+l}(2^{-s}\mathcal{L}_A).
\end{equation}
\end{thm}

\begin{proof} The claim is proved by induction on $k$ using the recursions
(\ref{3.15}). Again, we use the notation $\mathcal{L}:=2^{-s}\mathcal{L}_A.$
From (\ref{3.22}), it follows that
\begin{equation}\label{3.25}
T_{0,m}(\mathcal{L})=e^{\mathcal{L}+h_{m+l}(\mathcal{L})}.
\end{equation}
Furthermore, by using (\ref{3.13}) we infer from that
\begin{equation}\label{3.26}
T_{i,m}(\mathcal{L})=\mathcal{L}^{-i}\left(e^{\mathcal{L}+h_{m+l}(\mathcal{L})}-\sum\limits^{i-1}_{j=0}{\mathcal{L}}^j/j!\right),~i=1,2,\cdots,l.
\end{equation}
Now we assume that
\begin{eqnarray}\label{3.27}
T_{i,m}(2^{k-1}\mathcal{L})=(2^{k-1}\mathcal{L})^{-i}\left(e^{2^{k-1}\mathcal{L}+2^{k-1}h_{m+l}(\mathcal{L})}-\sum\limits^{i-1}_{j=0}(2^{k-1}\mathcal{L})^j/j!\right),~i=1,2,\cdots,l.
\end{eqnarray}
The inductive step follows from
\begin{eqnarray}\label{3.28}
\begin{aligned}
T_{i,m}(2^k\mathcal{L})&=\frac{1}{2^i} \left(T_{0,m}(\mathcal{L})^{2^{k-1}} T_{i,m}(2^{k-1}\mathcal{L})+ \sum\limits^{i}_{j=1}\frac{1}{(i-j)!}T_{j,m}(2^{k-1}\mathcal{L})\right)\\
=&(2^k\mathcal{L})^{-i}\left(e^{2^k\left(\mathcal{L}+h_{m+l}(\mathcal{L})\right)}
-e^{2^{k-1}\left(\mathcal{L}+h_{m+l}(\mathcal{L})\right)}\sum\limits^{i-1}_{j=0}(2^{k-1}\mathcal{L})^j/j!\notag\right.\\
&\left.+e^{2^{k-1}\left(\mathcal{L}+h_{m+l}(\mathcal{L})\right)}\sum\limits^{i}_{j=1}\frac{(2^{k-1}\mathcal{L})^{i-j}}{(i-j)!}
-\sum\limits^{i}_{j=1}\frac{1}{(i-j)!}\sum\limits^{j-1}_{\iota=0}\frac{(2^{k-1}\mathcal{L})^{i+\iota-j}}{\iota!}\right)\\
=&(2^k\mathcal{L})^{-i}\left(e^{2^k\left(\mathcal{L}+h_{m+l}(\mathcal{L})\right)}
-\sum\limits^{i-1}_{j=0}\sum\limits^{j}_{\iota=0}\frac{1}{(j-\iota)!\iota!}(2^{k-1}\mathcal{L})^j\right)\\
=&(2^k\mathcal{L})^{-i}\left(e^{2^k\left(\mathcal{L}+h_{m+l}(\mathcal{L})\right)}-\sum\limits^{i-1}_{j=0}(2^k\mathcal{L})^j/j!\right).
\end{aligned}
\end{eqnarray}
The desired (\ref{3.24}) follows by setting $k=s$.
\end{proof}
By comparing (\ref{3.24}) with (\ref{1.3c}) we note that $T_{l,m}(\mathcal{L}_A)$ generated by the recursion (\ref{3.15}) is a
perturbation of $\varphi_l(\mathcal{L}_A).$ The
perturbation term $\Delta \mathcal{L}_A$ can be regarded as a quasi-backward error and allow us to derive error bounds.
We want to ensure that
\begin{equation}\label{3.29}
\frac{\|\Delta \mathcal{L}_A\|}{\|\mathcal{L}_A\|}=\frac{\|h_{m+l}(2^{-s}\mathcal{L}_A)\|}{\|2^{-s}\mathcal{L}_A\|}\leq \sum\limits^\infty_{k=m+l}|c_{k+1}|\cdot\|(2^{-s}\mathcal{L}_A)^{k}\|\leq\text{Tol}
\end{equation}
for a given tolerance $\text{Tol}.$

Define $\bar{h}_{m+l}(x)=\sum\limits^\infty_{k=m+l}|c_{k+1}|x^{k}$ and let
\begin{equation}\label{3.30}
\theta_{m+l}=\max{\{\theta:{\bar{h}_{m+l}(\theta)}\leq \text{Tol}\}}.
\end{equation}
If $s$ is chosen such that
\begin{equation}\label{3.31}
2^{-s}\|{\mathcal{L}_A}^{k}\|^{1/k} \leq \theta_{m+l}~~\text{for}~~k>m+l,
\end{equation}
we have
\begin{equation}\label{3.32}
\frac{\|\Delta \mathcal{L}_A\|}{\|\mathcal{L}_A\|}\leq \text{Tol}.
\end{equation}
Table \ref{tab3.1} presents the values of $\theta_{m+l}$ satisfying the quasi-backward error bound (\ref{3.30}) for $\text{Tol}=2^{-53}$ for some values of $m+l$.
\begin{table}[h]
\setlength{\abovecaptionskip}{0.cm}
\setlength{\belowcaptionskip}{-0.3cm}
\caption{Some values of $\theta_{m+l}$ such that the quasi-backward error bound (\ref{3.30}) does not exceed $\text{Tol}=2^{-53}$.}\label{tab3.1}
\begin{center}
\begin{tabular*}{\textwidth}{@{\extracolsep{\fill}}@{~}c|ccccccccccccc}
\toprule%
{$m+l$} & $6$& $8$& $10$& $12$& $14$& $16$& $18$& $20$ & $22$ & $24$ & $26$ & $28$& $30$\\
  \hline
  $\theta_{m+l}$ & $9.1\text{e-}3$ & $5.0\text{e-}2$ & $1.4\text{e-}1$ & $3.0\text{e-}1$ &$5.1\text{e-}1$&$7.8\text{e-}1$&$1.1\text{e}0$
 &$1.4\text{e}0$&$1.8\text{e}0$&$2.2\text{e}0$&$2.6\text{e}0$&$3.1\text{e}0$&$3.5\text{e}0$ \\
\bottomrule
\end{tabular*}
\end{center}
\end{table}

It is not trivial to develop a cheap and suitable method to evaluate the quantities $\|{\mathcal{L}_A}^{k}\|^{1/k}$. Overestimation could cause a larger than necessary $s$ to be chosen, which will yield a negative effect on accuracy. For any $\epsilon>0$,
there exists a consistent norm $\|\cdot\|_\epsilon$ such that $\|\mathcal{L}_A\|_\epsilon \leq 2\rho (A)+\epsilon.$ It follows that
\begin{equation}
\|{\mathcal{L}_A}^k\|_\epsilon^{1/k} \leq 2\rho (A)+\epsilon.
\end{equation}
Thus, once
\begin{equation}
2^{-s}(2\rho (A)+\epsilon) \leq\theta_{m+l},
\end{equation}
we have
\begin{equation}\label{3.33}
\frac{\|\Delta \mathcal{L}_A\|_\epsilon}{\|\mathcal{L}_A\|_\epsilon}\leq \text{Tol}.
\end{equation}
In particular, if $A$ is normal, it is easily verified that $\|\mathcal{L}_A\|_F = 2\rho (A)$, and
 the bound (\ref{3.32}) then holds for the Frobenius norm if $2^{1-s}\rho (A) \leq\theta_{m+l}.$
Unfortunately, for non-normal matrix $A$ the bound (\ref{3.33}) described by the norm $\|\cdot\|_\epsilon$ is difficult
to interpret.

Now we present a quasi-backward error bound for general norm. Following Lemma 4.1 of \cite{AlMohy2009}, one can easily verify that
\begin{equation}
\|{\mathcal{L}_A}^k\|^{1/k} \leq \alpha _p(\mathcal{L}_A),~~p(p-1)\leq k,
\end{equation}
where $\alpha _p(\mathcal{L}_A)=\max(\|{\mathcal{L}_A}^{p}\|^{1/p},\|{\mathcal{L}_A}^{p+1}\|^{1/(p+1)}).$
Choose the parameter $s$ such that
\begin{equation}
2^{-s}\alpha _p(\mathcal{L}_A)\leq \theta_{m+l},~~p(p-1)\leq m+l,
\end{equation}
the quasi-backward error bound (\ref{3.32}) holds for any consistent norm.
For given $l$ and $m,$ the value of the scaling parameter $s$ is naturally chosen as
\begin{equation}\label{3.34}
s = max\{0, \lceil \text{log}_{2}(\alpha_{m+l}^*/\theta_{m+l})\rceil\},
\end{equation}
where $\alpha_{m+l}^*$ is the smallest value of $\alpha_p(\mathcal{L}_A)$ at which the value of $s$ is minimal.

This process requires pre-evaluating $\alpha_p(\mathcal{L})$ for $p(p-1)\leq m+l$, and thus
$\|{\mathcal{L}_A}^{p}\|^{1/p}$, $\|{\mathcal{L}_A}^{p+1}\|^{1/(p+1)}$ for $p(p-1)\leq m+l$.
However, evaluating these operator norm is a nontrivial task and has to be taken into account for the computational load.
A simple approach to evaluate $\|{\mathcal{L}_A}^{k}\|^{1/k}$ is to apply the formal power series of ${\mathcal{L}_A}^k[X]$. Direct calculation shows
\begin{equation}\label{3.35}
{\mathcal{L}_A}^k[X]=\sum\limits^{k}_{j=0}\mathcal{C}_k^jA^jX(A^{k-j})^T,
\end{equation}
where $\mathcal{C}_k^j:=\frac{k!}{j!(k-j)!}$ is the binomial coefficient.
From (\ref{3.35}) we have
\begin{equation}\label{3.36}
\|{\mathcal{L}_A}^k\|^{1/k}\leq d_k:=2\max\{\|A^j\|^{1/k}\cdot\|A^{k-j}\|^{1/k},~~ j=0,1,\ldots, k\}.
\end{equation}
Thus, the value of $\|{\mathcal{L}_A}^k\|^{1/k}$ can be replaced by the upper bound $d_k$.
We can use any consistent matrix norm but it is most convenient to use the 1-norm. As did in \cite{AlMohy2009},
we apply the block 1-norm estimation algorithm of Higham and Tisseur \cite{Higham00} to evaluate the 1-norm of the power of $A$ involved.

In practical, we choose the first $m\in \{6-l, 9-l, 12-l, 16-l, 20-l, 25-l~\}$ such that $\alpha_{m+l}^*\leq \theta_{m+l}$,
where $\alpha_{m+l}^*=\min\{\alpha_p(\mathcal{L}_A), p(p-1)\leq m+l\}$, and set $s=0$.
If $\alpha_{m+l}^*>\theta_{25}$, we set $m=25-l$ and $s=\lceil\log_2(\alpha_{25}^*/\theta_{25}), 0\rceil$.
The details on procedure for their choice are summarized in Algorithm \ref{alg3.1}.

\begin{algorithm}[htb]
\caption{this algorithm computes the parameters $m$ and $s$ by checking each putative $m$ such that the relative quasi-backward error can achieve prescribed accuracy.}\label{alg3.1}
\begin{algorithmic}[1]
\REQUIRE~$A,Q\in \mathbb{R}^{N\times N}$,~$l.$
\STATE $M=\{6,~  9,~  12,~  16,~  20,~  25\},$ ~$s=0,$ ~$p_{max}=5.$
\FOR {$p = 2:p_{\max}+1$}
\STATE Estimate $d_p =2\max{\left\{\|A^k\|_1\cdot\|A^{p-k}\|_1,~k=0,1,\ldots,p\right\}};$
\ENDFOR
\STATE Compute $\alpha_p=\max (d_p^{1/p}, ~d_{p+1}^{1/{(p+1)}}),$~$p=1,2,\cdots,p_{\max};$
\FOR {each $ m+l\in M$}
\STATE $\alpha^*_{m+l}=\min\{\alpha_p,~p(p-1)\leq m+l\}.$
\IF{$\alpha^*_{m+l}\leq\theta_{ l+m}$} \RETURN $m;$ \ENDIF
\ENDFOR
\STATE $s=\max(\lceil \log_2(\alpha^*_{m+l}/\theta_{m +l})\rceil,0).$
\ENSURE~$m,~s.$
\end{algorithmic}
\end{algorithm}

\section{Implementation issues}\label{sec:4}

This section describes some implementation details of the algorithm proposed in the above section.
A main problem is that the computational complexity of recursions (\ref{3.15}) grows exponentially with $s$,
which is mainly due to the approach of implementing the operator exponential $\varphi_0(2^{k}\mathcal{L})[\cdot]$ involved.

Alternatively, by using the identity (\ref{2.1a}), one can implement $\varphi_0(2^{k}\mathcal{L})[\cdot]$ by applying the following
coupled recurrences:
\begin{equation}\label{4.1}
\left\{
\begin{array}{l}
e^{2^k\tilde{A}}=e^{2^{k-1}\tilde{A}}\cdot e^{2^{k-1}\tilde{A}},\\
\varphi_0(2^k\mathcal{L})[\cdot]=e^{2^k\tilde{A}}[\cdot]e^{2^k\tilde{A}^T},
\end{array}
\right.
\end{equation}
where $\tilde{A}=2^{-s}A.$
This process requires computing the matrix exponential $e^{\tilde{A}}$ explicitly.
There are several established methods in the existing literature for carrying out this task, see e.g.,
\cite{AlMohy2009,Caliari19,Defez2018,Sastre19,Sastre2015,Higham2005,Sidje1998,Ward} and the review \cite{Moler2003}. Since the scaled matrix $\tilde{A}$ has small norm,
here we suggest approximating the matrix exponential using the order of $m+l$ truncated Taylor series
\begin{equation}\label{4.2}
e^{\tilde{A}}\approx T_{0,m}(\tilde{A}):= \sum\limits^{m+l}_{k=0}\frac{{\tilde{A}}^k}{k!}.
\end{equation}
The matrix polynomial $T_{0,m}(\tilde{A})$ can be computed efficiently by using the optimal Paterson-Stockmeyer method \cite{Paterson}.
If the value of $m+l$ is from the optimal index set $\mathbb{M}=\{2, 4, 6, 9, 12, 16, 20, 25, 30, 36,\ldots\},$ in which
the matrix polynomial $T_{0,m}(\tilde{A})$ will be the best approximation to $e^{\tilde{A}}$ at the same number of matrix-matrix multiplications,
the number of matrix-matrix product for evaluating $T_{0,m}(\tilde{A})$ is
\begin{equation}\label{4.3}
\pi_{m+l}=\left\lceil\sqrt {m+l}~\right\rceil+\left\lfloor \frac{m+l}{\lceil\sqrt {m+l}~\rceil}\right\rfloor-2.
\end{equation}
For more details see \cite[p. 72-74]{Higham}. A full sketch of the procedure for solving $\varphi_l(\mathcal{L}_A)[Q]$ is summarized in Algorithm \ref{alg4.1}.

It is clear that the matrix-matrix multiplications constitute the main cost of Algorithm \ref{alg4.1}
since the rest of the required operations is limited to several matrix additions and scalar-matrix multiplications.
All together, the total number of matrix-matrix multiplications $C_{l,m}$ required to evaluate $T_{l,m}(\mathcal{L}_A)[Q]$ is
\begin{equation}\label{2.21}
C_{l,m}:=\left\{
\begin{array}{l}
m,~~s=0,\\
\pi_{m+l}+m+l+1+(s-1)(2l+1),~~s\geq1,\\
\end{array}
\right.
\end{equation}
where $\pi_{m+l}$ is defined as (\ref{4.3}).

\begin{algorithm}[htb]
\caption{~$\textsf{philyap}$: this algorithm computes $\varphi_l(\mathcal{L}_A)[Q]$ based on the modified scaling and squaring combination with Taylor series.}\label{alg4.1}
\begin{algorithmic}[1]
\REQUIRE $A, Q \in \mathbb{R}^{N\times N},$ $l$;
\STATE Select the values of $m$ and $s$ using Algorithm \ref{alg3.1};
\STATE Compute $\tilde{A}=2^{-s}A;$
\STATE Compute $T_{l,m}(\mathcal{L}_{\tilde{A}})[Q]$ by Horner's method;
\IF {$s=0$} \RETURN $T_{l,m}(\mathcal{L}_{\tilde{A}})[Q];$ \ENDIF
\FOR{$k=l-1:1$}
\STATE Compute $T_{k,m}(\mathcal{L}_{\tilde{A}})[Q]:=\mathcal{L}_{\tilde{A}}\left[T_{k+1,m}(\mathcal{L}_{\tilde{A}})[Q]\right]+\frac{1}{k!}Q;$
\ENDFOR
\STATE Compute $T_{0,m}(\tilde{A})$ by the Paterson-Stockmeyer method;
\FOR{$i=1:s-1$}
\STATE Compute $T_{k,m}(2^i\mathcal{L}_{\tilde{A}})[Q]:=\frac{1}{2^k} \left(T_{0,m}(2^{i-1}\tilde{A})\cdot T_{k,m}(2^{i-1}\mathcal{L}_{\tilde{A}})[Q]\cdot T_{0,m}(2^{i-1}\tilde{A})^T+ \sum\limits^{k}_{j=1}\frac{1}{(k-j)!}T_{j,m}(2^{i-1}\mathcal{L}_{\tilde{A}})[Q]\right),~k=1,\cdots,l$;
\STATE Compute $T_{0,m}(2^{i}\tilde{A}):=T_{0,m}(2^{i-1}\tilde{A})\cdot T_{0,m}(2^{i-1}\tilde{A})$;
\ENDFOR
\STATE Compute $T_{l,m}(\mathcal{L}_A)[Q]=\frac{1}{2^l} \left(T_{0,m}(2^{s-1}\tilde{A})\cdot T_{l,m}(2^{s-1}\mathcal{L}_{\tilde{A}})[Q]\cdot T_{0,m}(2^{s-1}\tilde{A})^T+ \sum\limits^{l}_{j=1}\frac{1}{(l-j)!}T_{j,m}(2^{s-1}\mathcal{L}_{\tilde{A}})[Q]\right)$;
\ENSURE~$T_{l,m}(\mathcal{L}_A)[Q].$
\end{algorithmic}
\end{algorithm}

\section{Numerical experiments}
In this section we present a few numerical experiments to test the performance of the method that has been presented in Section \ref{sec:4}.
All the tests are performed under Windows 10 and MATLAB R2018b running on a desktop with an Intel Core i7 processor with 2.1 GHz and RAM 64 GB.
The relative error is measured in the 1-norm, i.e.,
\begin{equation}\label{5.1}
Error= \frac{\|Y-\widehat{Y}\|_1}{\|Y\|_1},
\end{equation}
where $\widehat{Y}$ is the computed solution and $Y$ is the reference solution.

To benchmark our method, in the first two experiments we have run comparison tests with some MATLAB functions tailored for the matrix $\varphi$-function,
since the computation of operator $\varphi$-function is mathematically equivalent to computing the action of a matrix $\varphi$-function on vector by (\ref{2.6}).
The codes involved are listed as follows.

$\bullet$ The MATLAB function \textsf{expmv} of Al-Mohy and Higham \cite{AlMohy2011} computes the action of
matrix exponential on a vector based on matrix-vector products.
The function can be utilized to evaluate the matrix $\varphi$-functions by computing a slightly larger matrix exponential.

$\bullet$ The MATLAB function \textsf{phimv(s)} of Li, Yang and Lan \cite{Li2022} computes the action of the matrix $\varphi$-functions
on a vector. The method is an implementation of the modified scaling and squaring procedure combined with a truncated Taylor series.

$\bullet$ \textsf{kiops} is the MATLAB function due to Gaudreault, Rainwater and Tokman \cite{Tokman18}, which
computes a linear combination of $\varphi$-functions acting on certain vectors using Krylov-based method combined with time-stepping.
It can be viewed as an improved version of \textsf{phipm} proposed in \cite{Niesen2012}.

Unless otherwise stated, we run all these MATLAB functions with their default parameters and the
convergence tolerance in every algorithm is set to the machine epsilon $2^{-53}.$

\begin{example} \label{exa1}
In the first experiment, we try to show the performance of \textsf{philyap} by using sixty one Lyapunov operators $\mathcal{L}_A$.
The first 47 operators are generated by matrices of size $8\times 8$ from the subroutine \textsf{matrix} in the Matrix Computation Toolbox \cite{Highamtool}.
 The other fourteen operators are generated by matrices of dimensions $2-10$ from \cite[Ex. 2]{Higham2003}, \cite[Ex. 3.10]{DP00},
 \cite[p. 655]{KL1998}, \cite[p. 370]{NH1995}, \cite[Test Cases 1-4]{Ward}, respectively.
For each Lyapunov operator $\mathcal{L}_A$, and a different randomly generated symmetric matrix $Q$ for each $\mathcal{L}_A$,
we compute $\varphi_l(\mathcal{L}_A)[Q]$ for $l=1,2,\ldots,8.$ The implementations are compared with MATLAB functions \textsf{expmv} and \textsf{phimv(s)}.
In this experiment, the reference solutions are computed using the function \textsf{phipade} from the software package EXPINT \cite{Berland07}
based on [17/17] Pad\'e approximation at 100-digit precision using MATLAB's Symbolic Math Toolbox.

Figs. \ref{fig5.1} and \ref{fig5.2} present the relative errors and the performances on execution times of the three methods, respectively.
 Each figure contains eight plots, which correspond with the results for computing $\varphi_l(\mathcal{L}_A)[Q]$ for $l=1,2,\ldots,8.$

From Fig. \ref{fig5.1} we see that the method \textsf{philyap} can perform in a numerically similar way with \textsf{phimv(s)} and it is generally more accurate than \textsf{expmv}.
Table \ref{tab5.1} lists the percentage of cases in which the relative errors of \textsf{philyap} are lower than the relative errors of MATLAB codes
 \textsf{expmv} and \textsf{phimv}. Results show that \textsf{philyap} is more accurate than \textsf{phimv(s)} and \textsf{expmv} in the majority of cases.

\begin{figure}[H]
\begin{minipage}{0.5\linewidth}
\centering
\includegraphics[width=7cm,height=5cm]{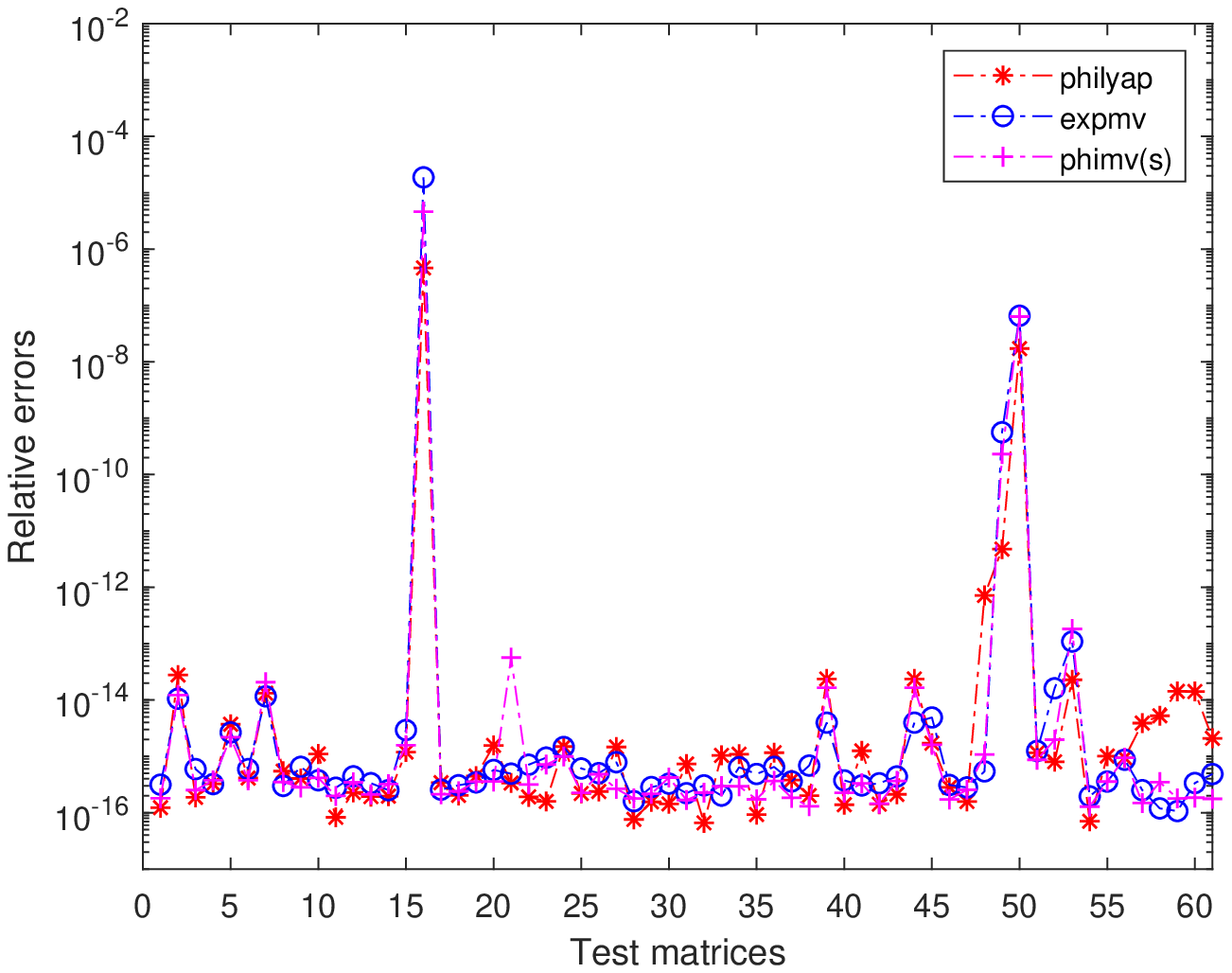}\\
\end{minipage}
\mbox{\hspace{-1.5cm}}
\begin{minipage}{0.5\linewidth}
\centering
\includegraphics[width=7cm,height=5cm]{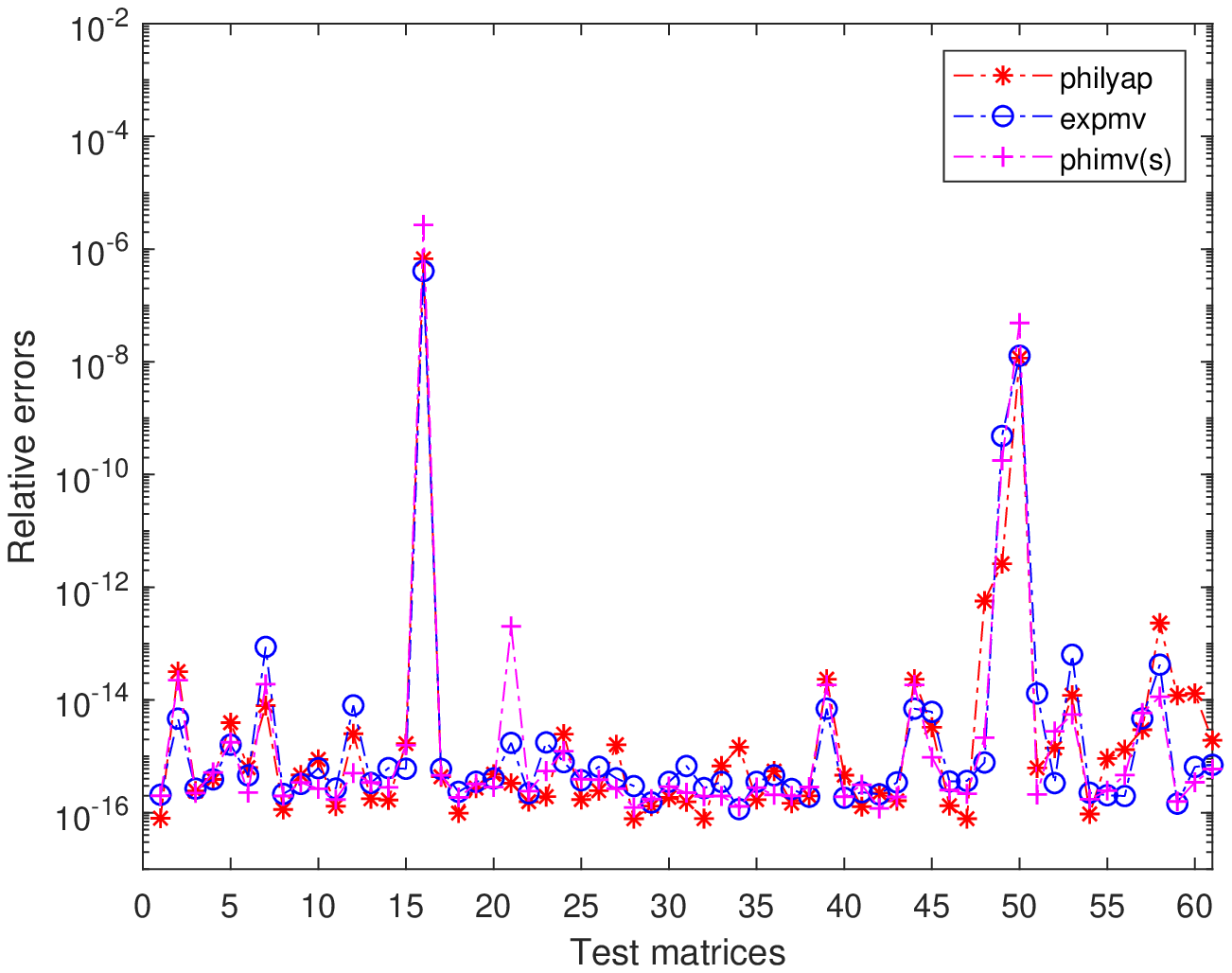}\\
\end{minipage}\\
\begin{minipage}{0.5\linewidth}
\centering
\includegraphics[width=7cm,height=5cm]{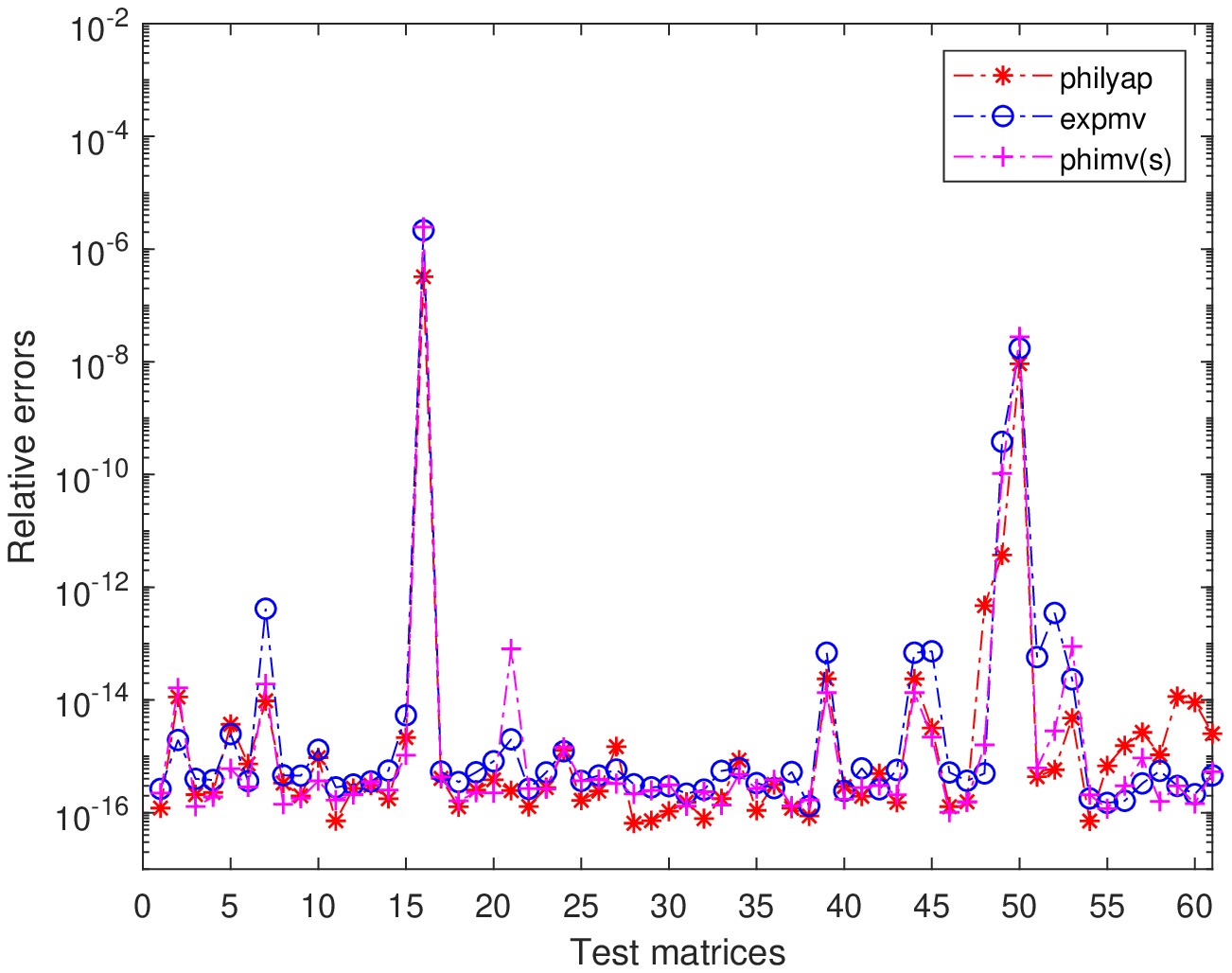}\\
\end{minipage}
\mbox{\hspace{-1.5cm}}
\begin{minipage}{0.5\linewidth}
\centering
\includegraphics[width=7cm,height=5cm]{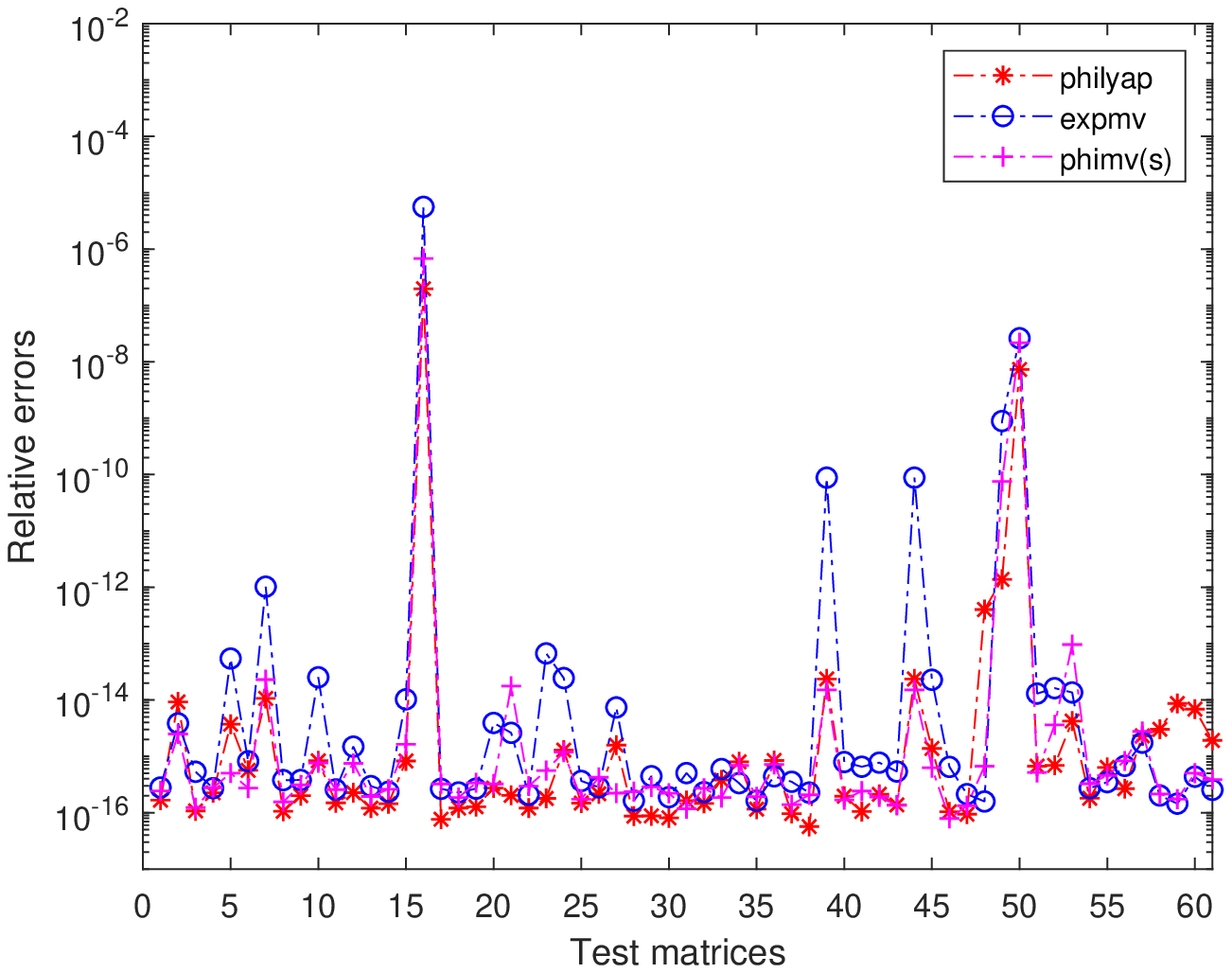}\\
\end{minipage}\\
\begin{minipage}{0.5\linewidth}
\centering
\includegraphics[width=7cm,height=5cm]{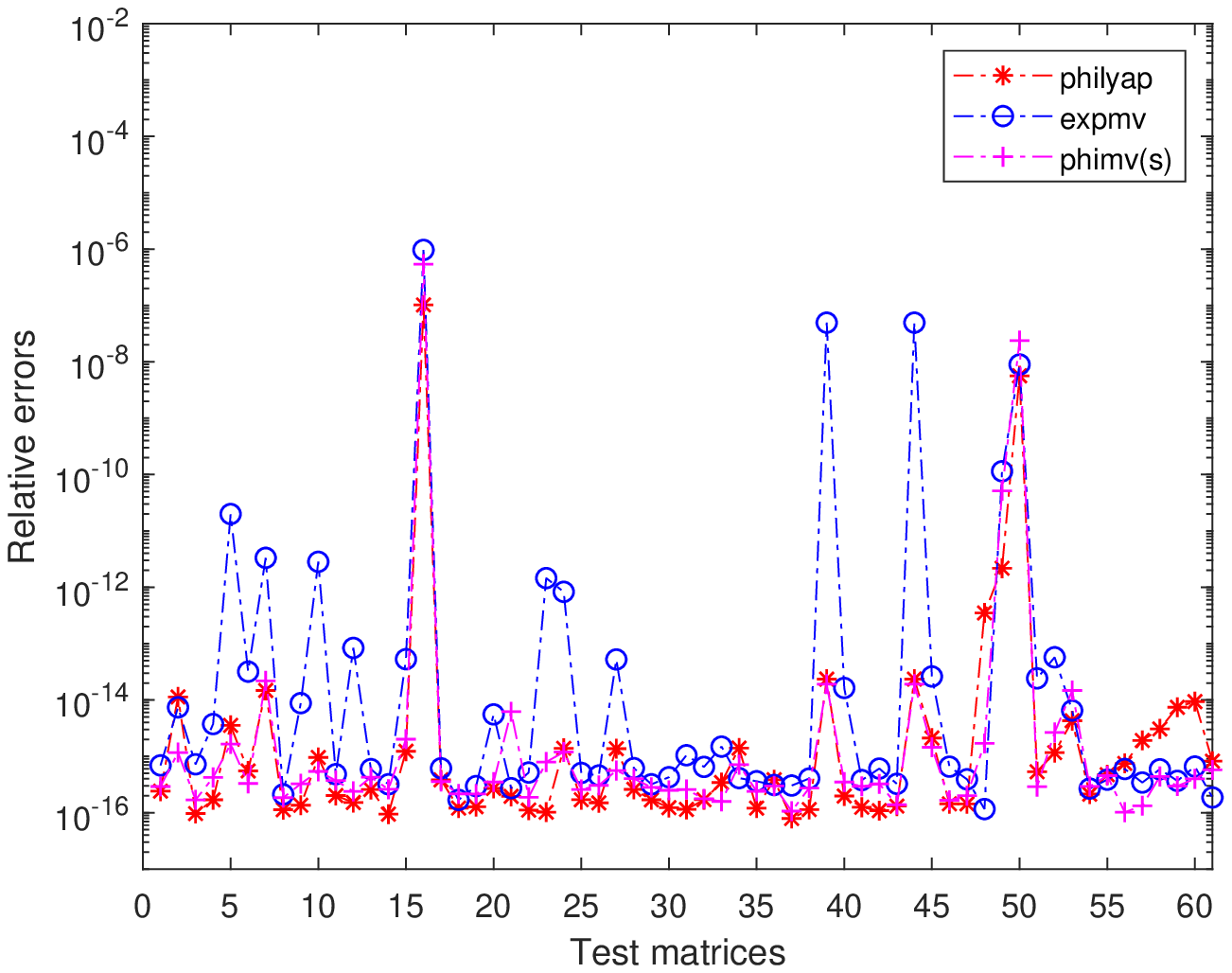}\\
\end{minipage}
\mbox{\hspace{-1.5cm}}
\begin{minipage}{0.5\linewidth}
\centering
\includegraphics[width=7cm,height=5cm]{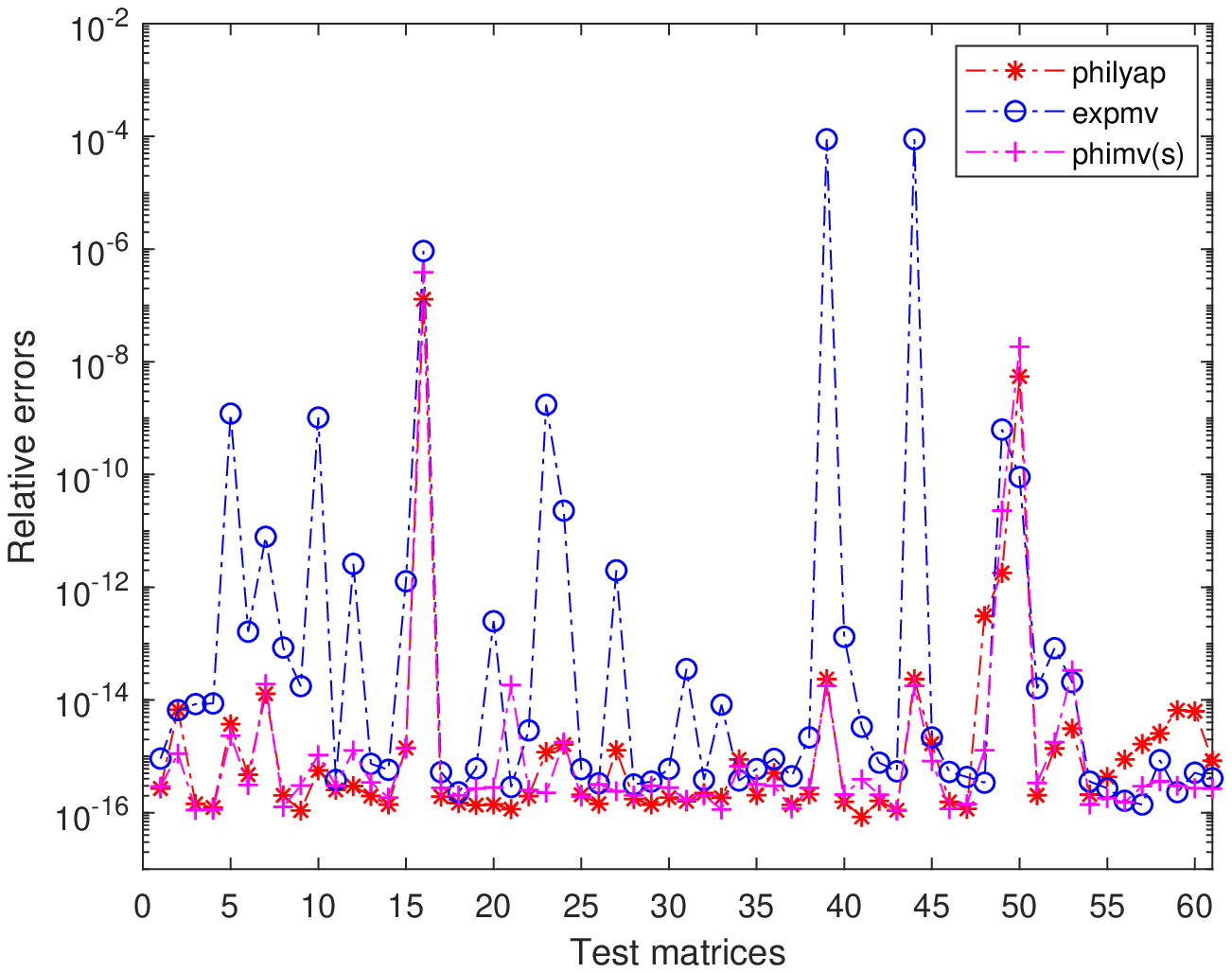}\\
\end{minipage}\\
\begin{minipage}{0.5\linewidth}
\centering
\includegraphics[width=7cm,height=5cm]{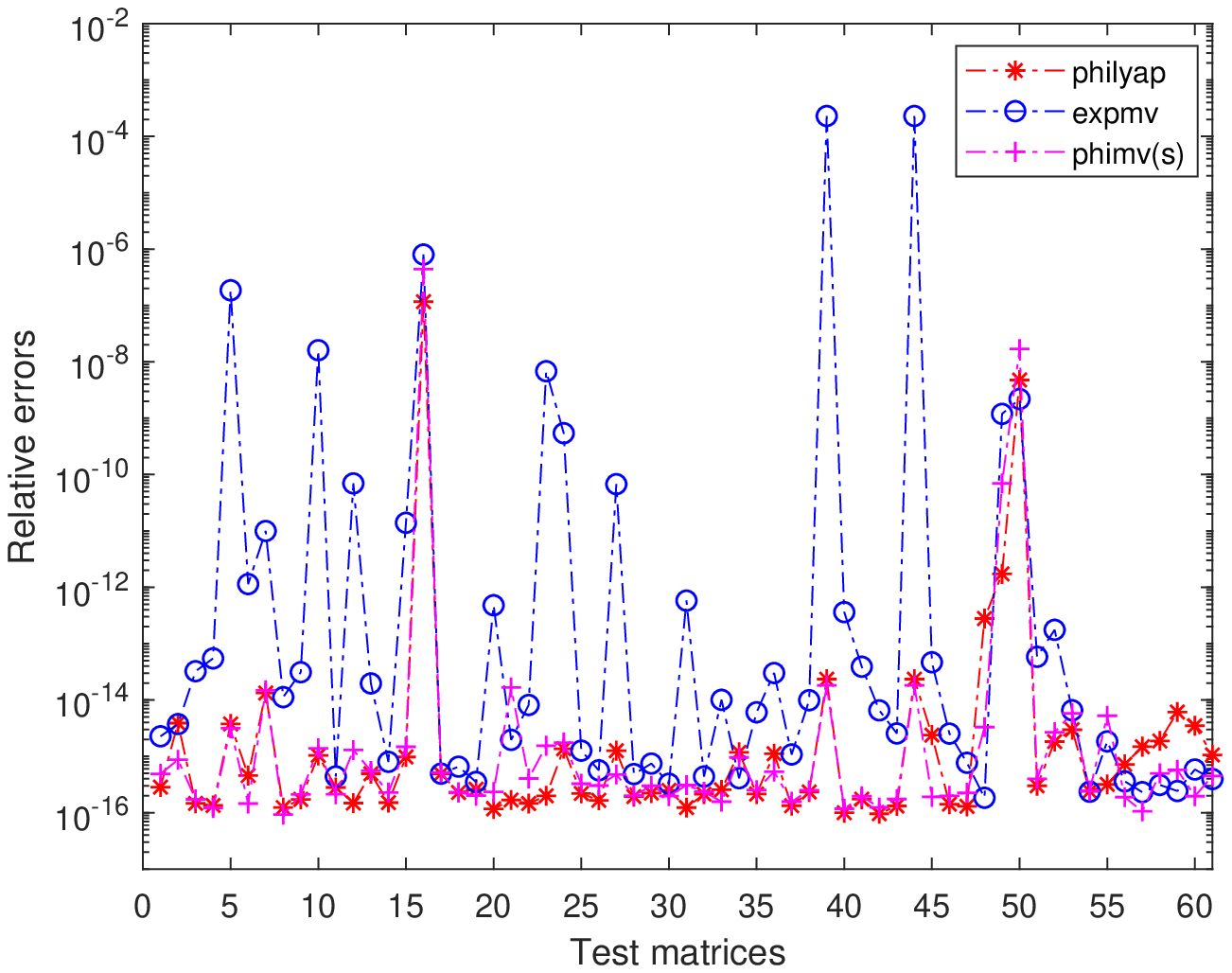}\\
\end{minipage}
\mbox{\hspace{-1.5cm}}
\begin{minipage}{0.5\linewidth}
\centering
\includegraphics[width=7cm,height=5cm]{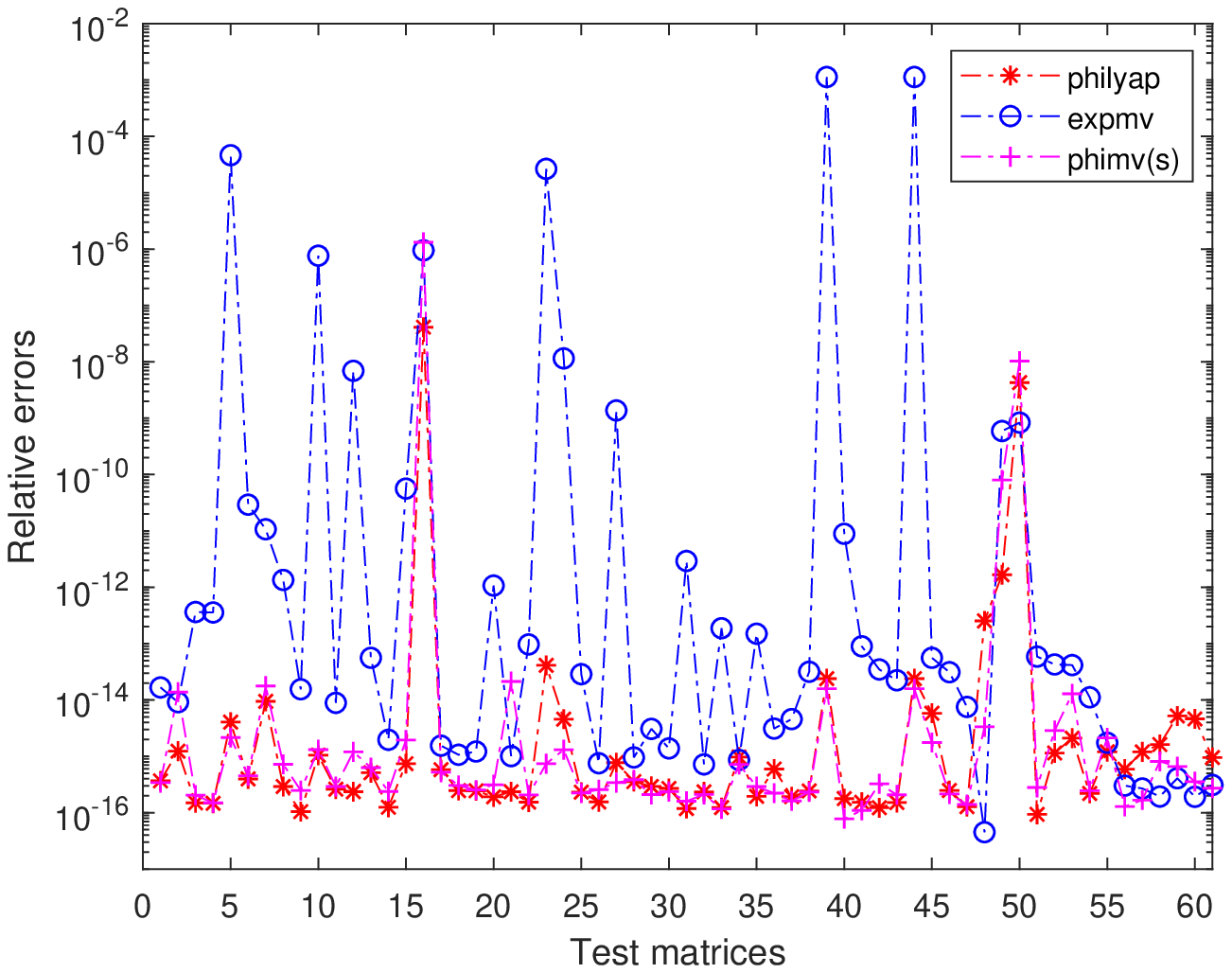}\\
\end{minipage}
\caption{ From left to right we plot the relative errors when computing $\varphi_l(\mathcal{L}_A)[Q]$ for $l= 1, . . . , 8$ of Experiment \ref{exa1}.}\label{fig5.1}
\end{figure}

\begin{figure}[H]
\begin{minipage}{0.5\linewidth}
\centering
\includegraphics[width=7cm,height=5cm]{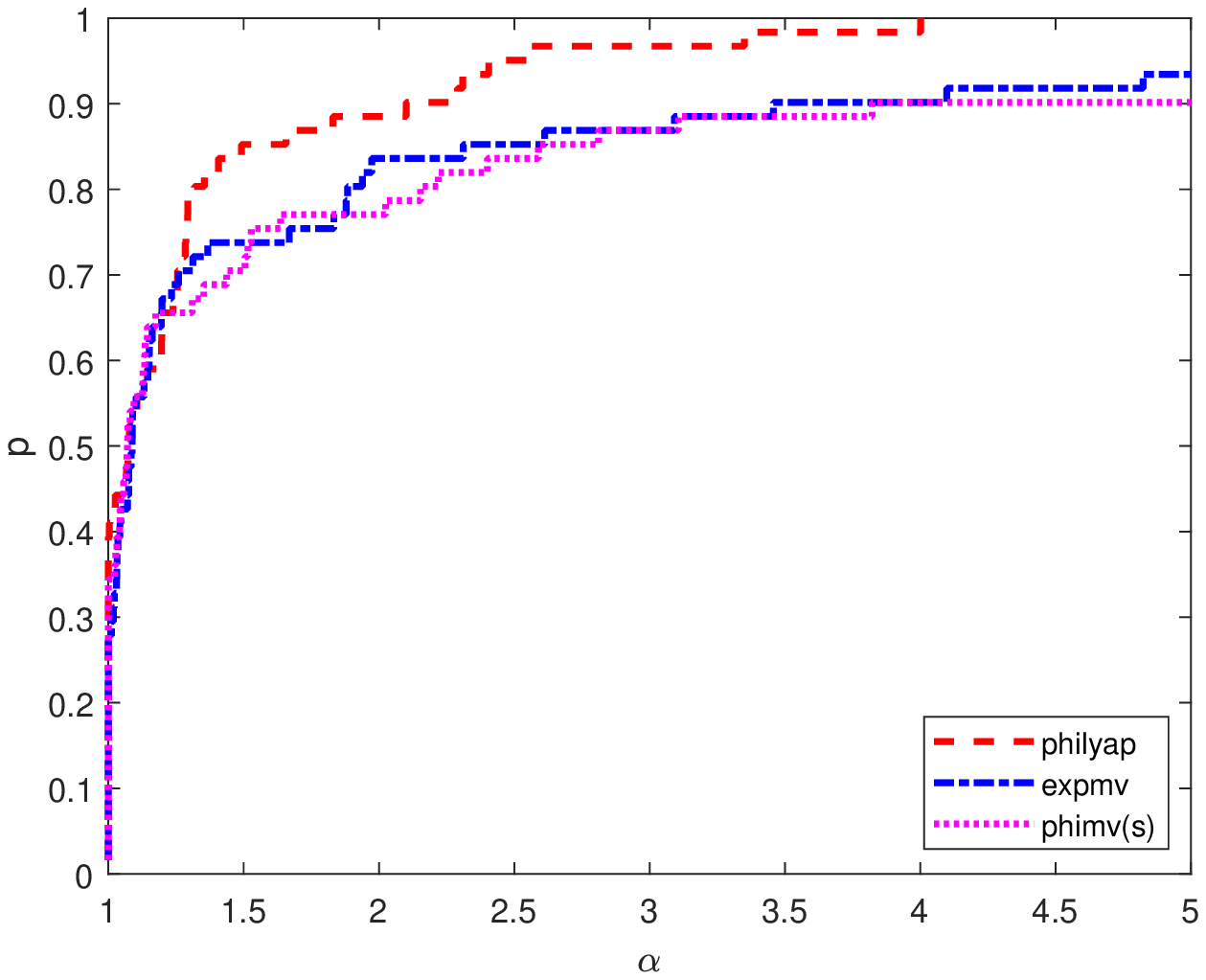}\\
\end{minipage}
\mbox{\hspace{-1.5cm}}
\begin{minipage}{0.5\linewidth}
\centering
\includegraphics[width=7cm,height=5cm]{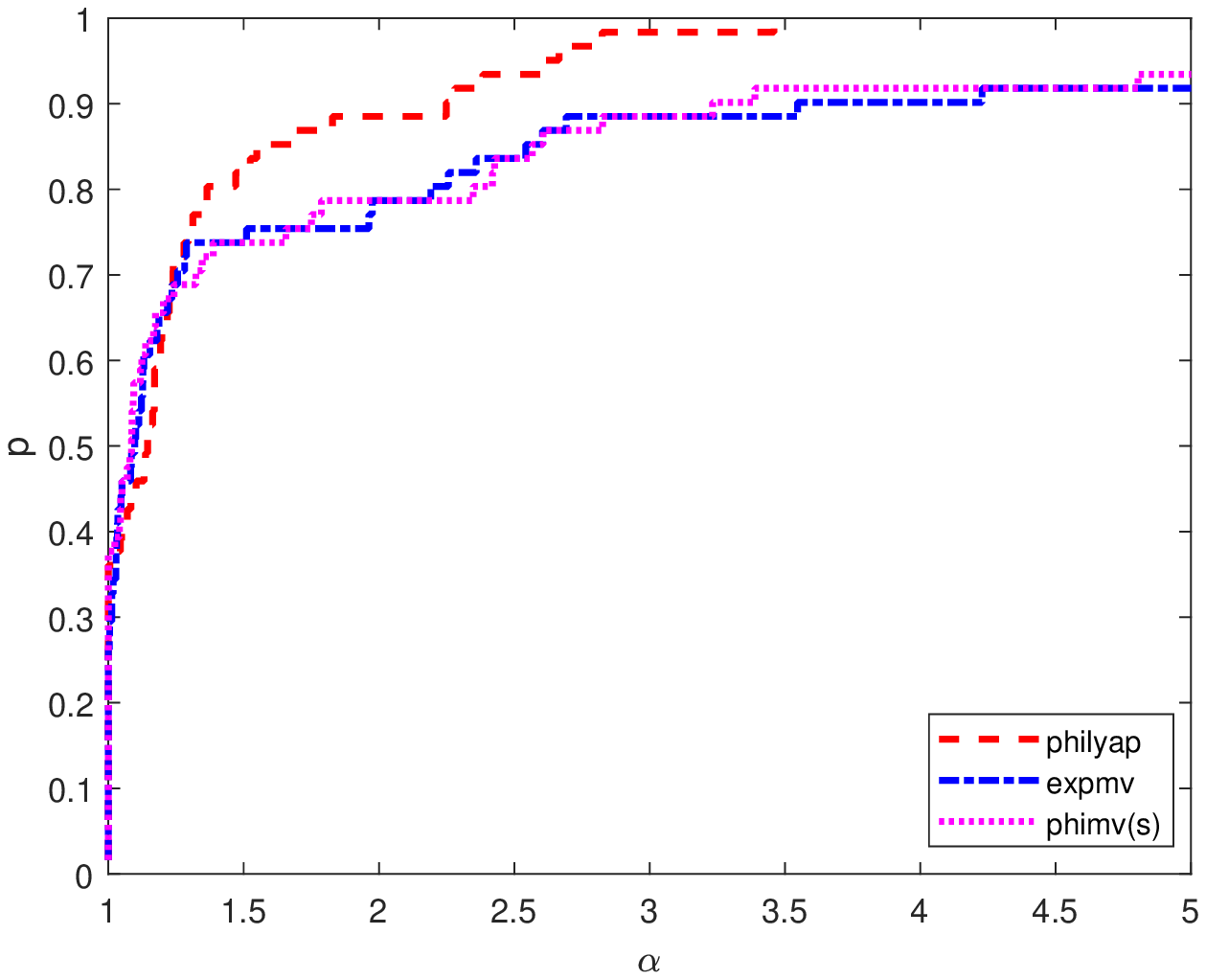}\\
\end{minipage}\\
\begin{minipage}{0.5\linewidth}
\centering
\includegraphics[width=7cm,height=5cm]{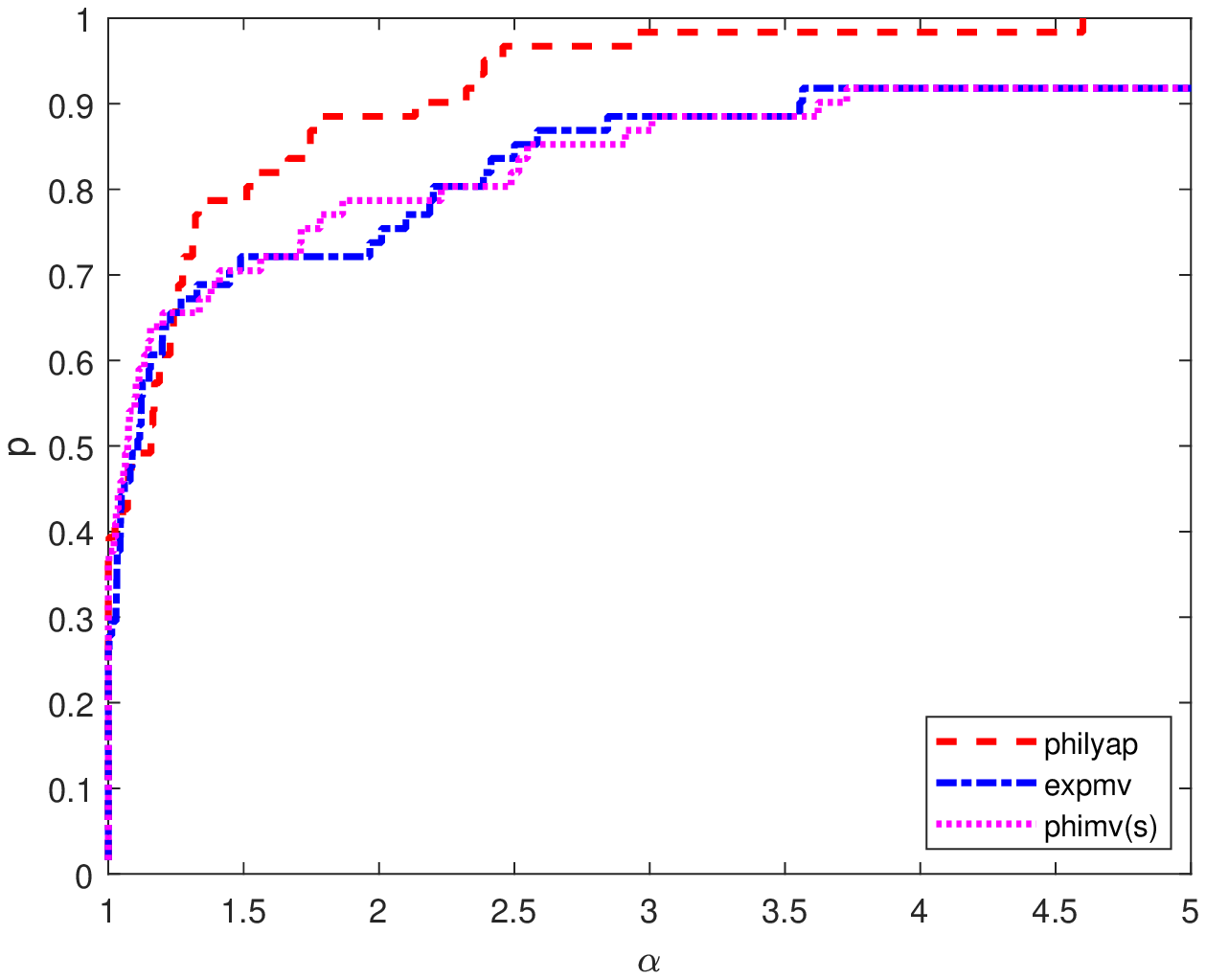}\\
\end{minipage}
\mbox{\hspace{-1.5cm}}
\begin{minipage}{0.5\linewidth}
\centering
\includegraphics[width=7cm,height=5cm]{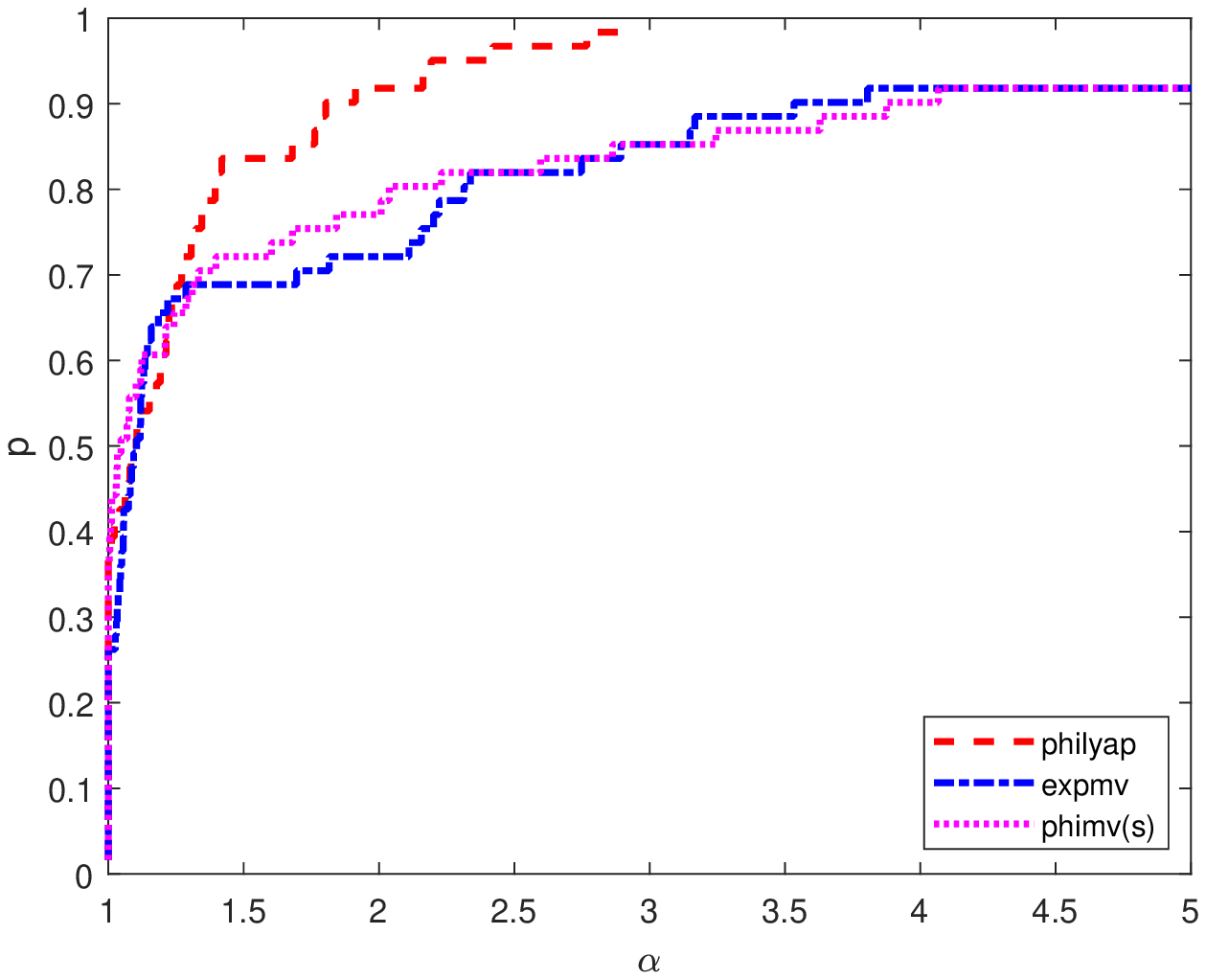}\\
\end{minipage}\\
\begin{minipage}{0.5\linewidth}
\centering
\includegraphics[width=7cm,height=5cm]{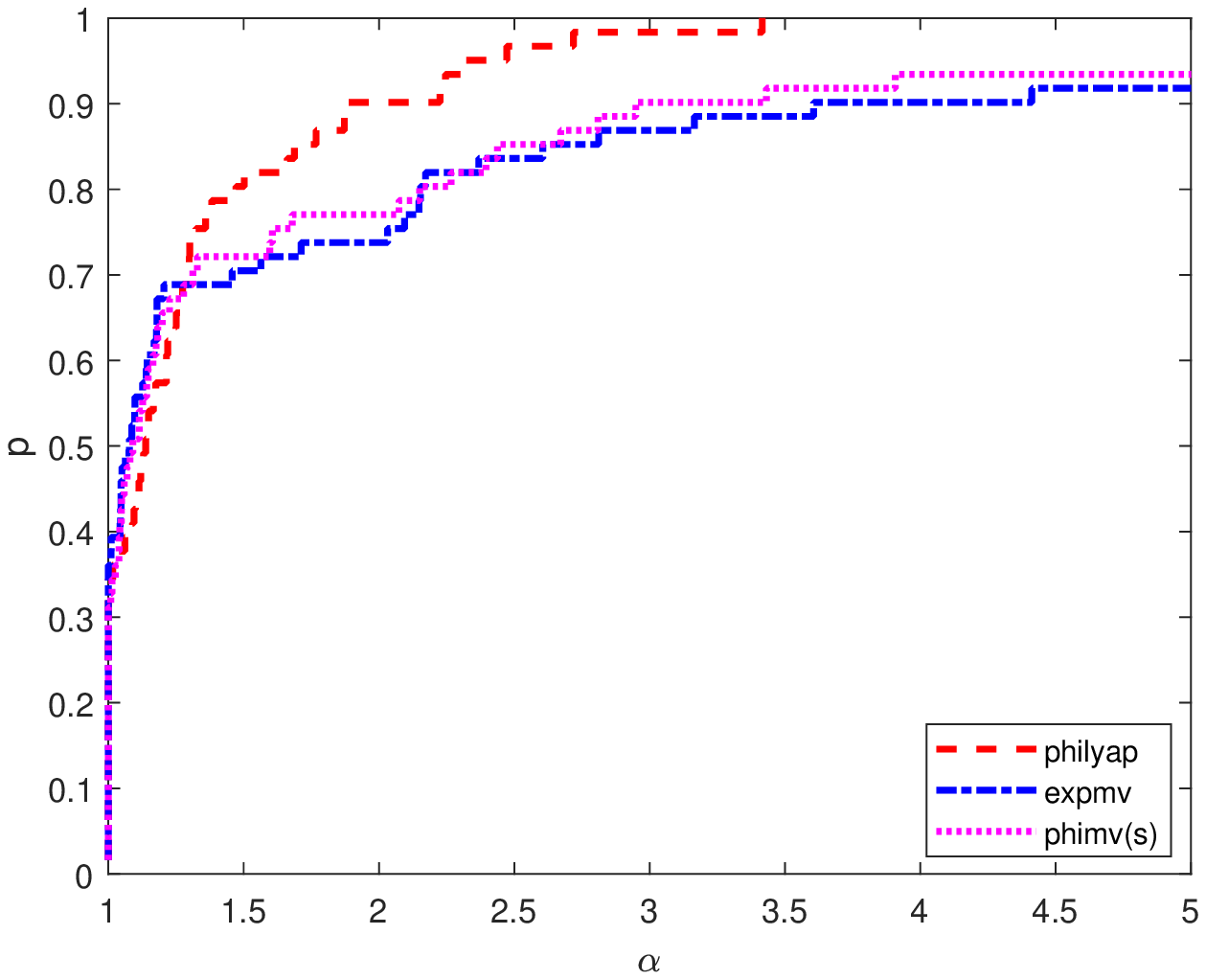}\\
\end{minipage}
\mbox{\hspace{-1.5cm}}
\begin{minipage}{0.5\linewidth}
\centering
\includegraphics[width=7cm,height=5cm]{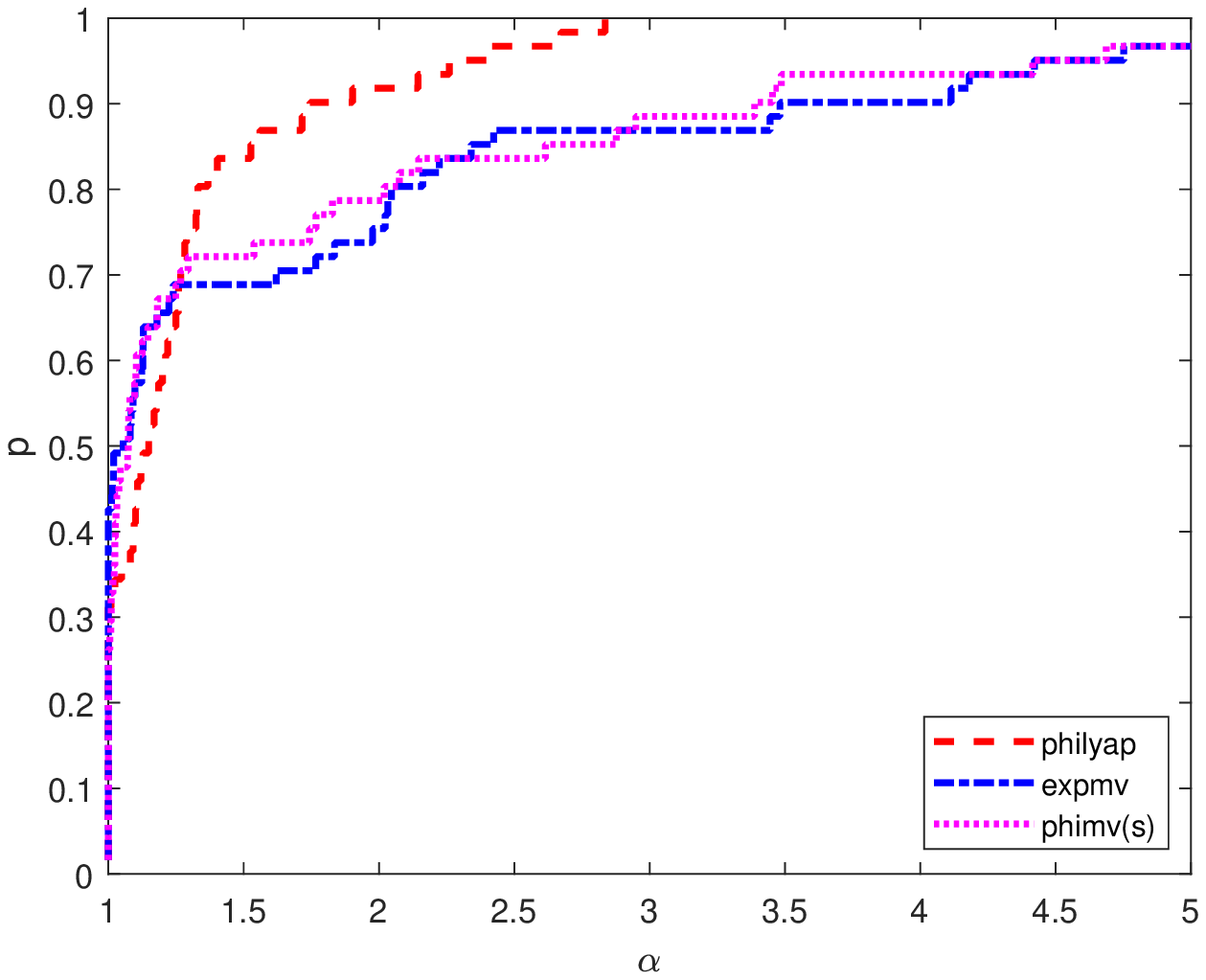}\\
\end{minipage}\\
\begin{minipage}{0.5\linewidth}
\centering
\includegraphics[width=7cm,height=5cm]{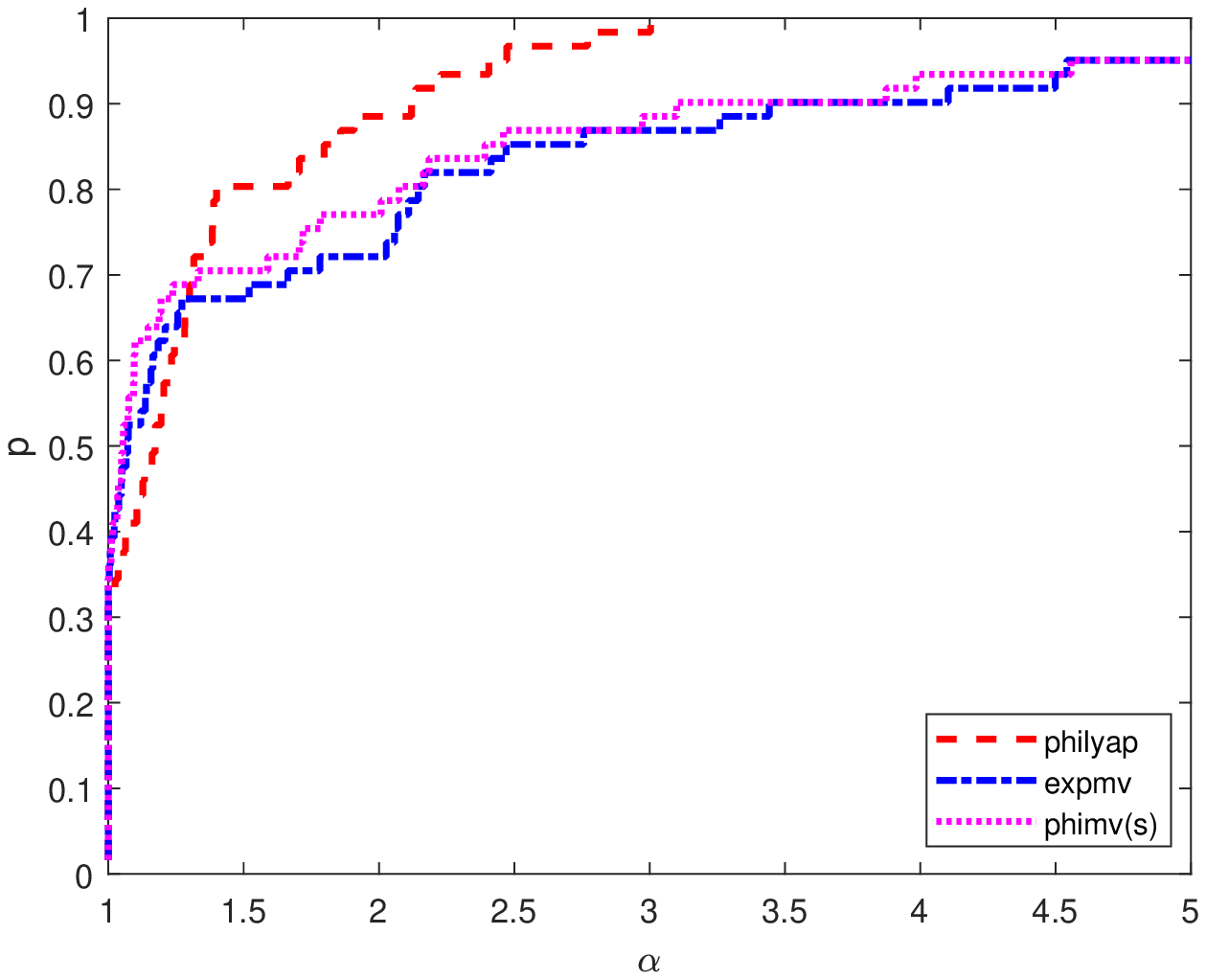}\\
\end{minipage}
\mbox{\hspace{-1.5cm}}
\begin{minipage}{0.5\linewidth}
\centering
\includegraphics[width=7cm,height=5cm]{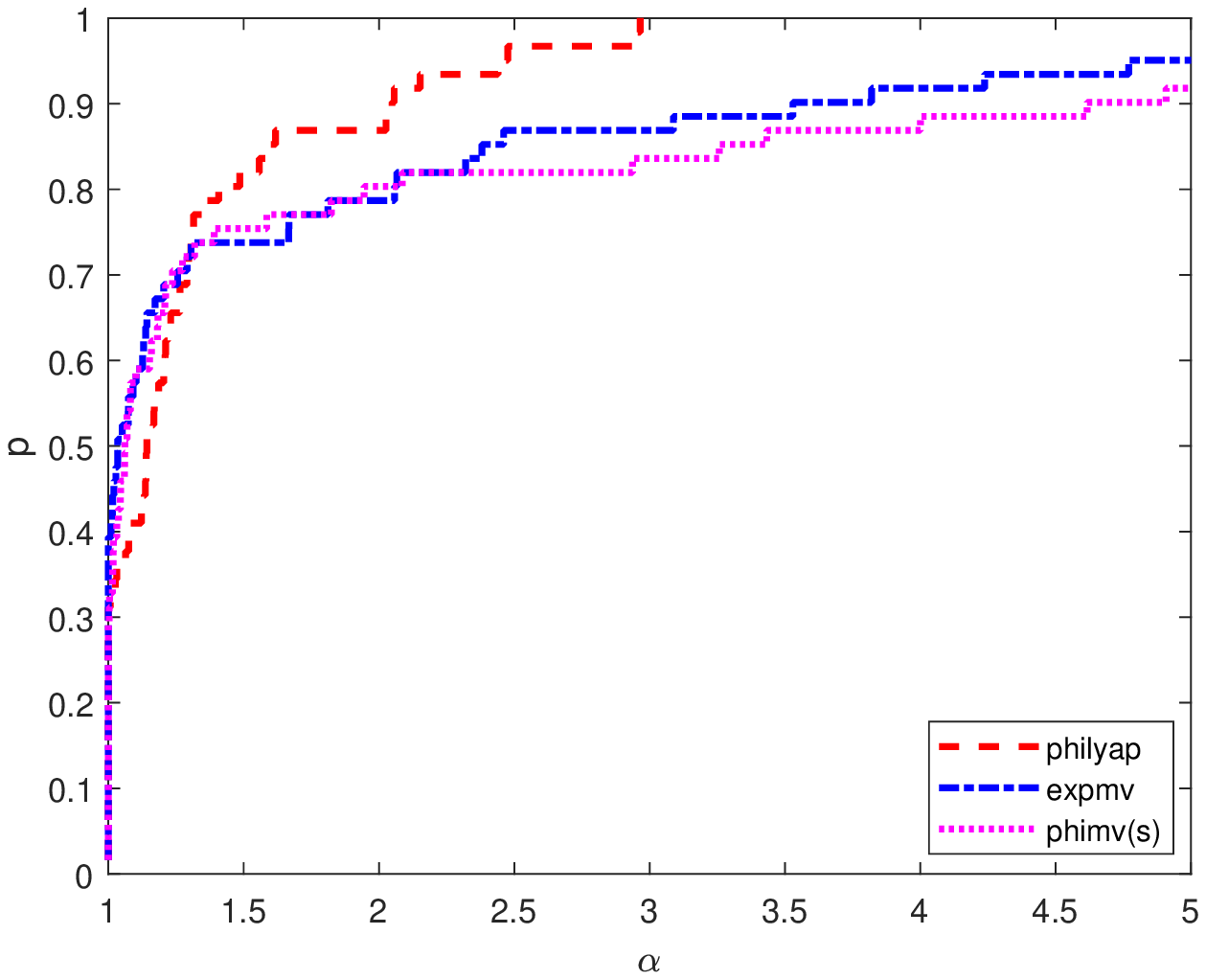}\\
\end{minipage}
\caption{ From left to right we plot the performance of execution times when computing $\varphi_l(\mathcal{L}_A)[Q]$ for $l= 1, . . . , 8$
 of Experiment \ref{exa1}.}\label{fig5.2}
\end{figure}

Fig. \ref{fig5.2} shows the performance profiles for the test set, where for a given $\alpha$ the corresponding value of $p$ on each performance
curve is the fraction that the considered method spends a time within a factor $\alpha$ of the least time over all the methods involved \cite{DM02}.
We see that for this test set \textsf{philyap} performs better than the other two MATLAB
functions. In a preliminary analysis, we tested the performance of the three methods for one hundred operators $\mathcal{L}_A$ generated by random matrices
of size $128\times 128$. We found that the execution time of \textsf{philyap} is obviously smaller than the other two methods.
Due to the computation of the reference solutions are extremely time consuming and we have not therefore reported here.
\end{example}

\begin{table}[h]
\setlength{\abovecaptionskip}{0.cm}
\setlength{\belowcaptionskip}{-0.3cm}
\caption{Percentage of times that the relative error of \textsf{philyap} is lower than \textsf{expmv} and \textsf{phimv(s)} for Experiment \ref{exa1}.}\label{tab5.1}
\begin{center}
\begin{tabular*}{\textwidth}{@{\extracolsep{\fill}}@{~}c|cccccccc}
\toprule%
{l} & $1$& $2$& $3$& $4$& $5$& $6$& $7$& $8$ \\
  \hline
  $E(\textsf{philyap})<E(\textsf{expmv})$ & $56\%$ & $57\%$ & $72\%$ & $84\%$ &$82\%$&$82\%$
  &$82\%$&$85\%$ \\
 \hline
 $E(\textsf{philyap})<E(\textsf{phimv(s)})$ & $49\%$ & $54\%$ & $49\%$ & $61\%$ &$64\%$&$52\%$
 &$64\%$&$54\%$ \\
\bottomrule
\end{tabular*}
\end{center}
\end{table}

\begin{example} \label{exa2}
In the experiment we compare \textsf{philyap} with \textsf{expmv} and \textsf{kiops} by evaluating $\varphi_l(\mathcal{L}_A)[Q]$ for $l=1,2,\ldots,8$,
and show the efficiency of our new algorithm.
Let $\mathcal{L}_A$ be generated by the tridiagonal matrix $A=2500\cdot diag(1,-2,1)\in\mathbb{R}^{400\times 400}$,
and let $Q\in \mathbb{R}^{400\times 400}$ be randomly symmetric matrix.
The Lyapunov operator $\mathcal{L}_A$ can be naturally regarded as the the standard 5-point difference discretization
of the two-dimensional Laplacian operator $\frac{\partial X}{\partial x}+\frac{\partial X}{\partial y}$ on the unit
square with $400$ nodes in each spatial dimension.
The reference solutions are obtained from running MATLAB built-in function \textsf{ode45} with absolute tolerance of $10^{-20}$ and
relative tolerance of $2.22045\cdot10^{-14}$. These have been done by vectorizing the corresponding LDEs (\ref{2.5})
into a vector-valued ODEs with $160000$ unknowns.

Table \ref{tab5.2} lists the performance of the three methods in terms of both accuracy and CPU time.
We also list the the execution time of \textsf{ode45} in the last column. We note that the
errors obtained with each one are similar but the execution time, however, is obviously smaller when \textsf{philyap} is used.
\end{example}

\begin{table}[H]
\setlength{\abovecaptionskip}{0.cm}
\setlength{\belowcaptionskip}{-0.3cm}
\caption{The CPU time (in seconds) and the relative errors when computing $\varphi_l(\mathcal{L}_A)[Q]$ for $l=1,2,\ldots,8$ of Experiment \ref{exa2}.}
 \label{tab5.2}
\begin{center}
\begin{tabular*}{\textwidth}{@{\extracolsep{\fill}}@{~~}c|lr|lr|lr|c}
\toprule%
\raisebox{-2.00ex}[0cm][0cm]{$l$}&
\multicolumn{2}{c|}{\textsf{expmv}}&\multicolumn{2}{c|}{\textsf{kiops}}&\multicolumn{2}{c|}{\textsf{philyap}}&\multicolumn{1}{c}{\textsf{ode45}}\\
\cline{2-8}
 &error &time&error&time &error&time & time\\
\midrule%
\multirow{1}{*}{1}
&9.7950e-14&131.01    &8.9892e-15&16.09   &3.8019e-14&0.18 &261.12     \\\hline
\multirow{1}{*}{2}
&2.5301e-13&130.18   &6.8154e-15&15.82   &2.3683e-14&0.24 &165.26   \\\hline
\multirow{1}{*}{3}
&1.4198e-13&130.24    &5.2936e-14&15.15   &1.7568e-14&0.32 & 2137.38   \\\hline
\multirow{1}{*}{4}
&2.1149e-13&129.73    &3.8412e-14&13.80   &1.3858e-14& 0.39 & 1809.65  \\\hline
\multirow{1}{*}{5}
&3.4152e-13&129.66   &2.8272e-14&13.01   &1.1563e-14&0.49  &1391.77  \\\hline
\multirow{1}{*}{6}
&5.9342e-15&130.56    &4.0305e-14&12.28   &1.0012e-14&0.59 &975.93   \\\hline
\multirow{1}{*}{7}
&4.3692e-13&129.77    &7.3940e-14&12.22   &8.8777e-15&0.68 &673.22   \\\hline
\multirow{1}{*}{8}
&6.2239e-14&130.69    &8.6059e-15&12.49  &8.2295e-15&0.83&450.93   \\
\bottomrule
\end{tabular*}
\end{center}
\end{table}

\begin{example} \label{exa3}
To illustrate the behavior of the matrix-valued exponential integrators implemented with the function \textsf{philyap},
we consider the differential Riccati equations :
\begin{equation}\label{5.2}
\left\{
\begin{array}{l}
X'(t)=AX(t)+X(t)A^T+CC^T-X(t)BB^TX(t)=:F(X(t)),\\
X(0)=I,
\end{array}
\right.
\end{equation}
where the matrix $A\in\mathbb{R}^{400\times 400}$ stems from the spatial finite difference discretization of the following advection-diffusion model
\begin{equation*}
\frac{\partial u}{\partial t}=\Delta u - 10x \frac{\partial u}{\partial x} - 100y\frac{\partial u}{\partial y}
\end{equation*}
on the domain $\Omega=(0,1)^2$ with homogeneous Dirichlet boundary conditions, and $B, C^T\in \mathbb{R}^{400\times 1}$ are the corresponding load vectors.
The system matrices $A,$ and $B,$ $C$ can be generated directly by MATLAB functions
\textsf{fdm\underline{~}2d\underline{~}matrix}
and \textsf{fdm\underline{~}2d\underline{~}vector},
respectively, from LYAPACK toolbox \cite{Penzl}. This is a widely used test system.
We integrate system (\ref{5.2}) with $B=\textsf{fdm\underline{~}2d\underline{~}vector}(20,'.1<x<=.3')$
and $C=\textsf{fdm\underline{~}2d\underline{~}vector}(20, '.7<x<=.9')$
using two matrix-valued exponential Rosenbrock-type integration schemes \textsf{exprb2} and \textsf{exprb3} presented in \cite{Li2021}.
The first scheme is of order two and requires the computation of the first operator $\varphi$-function. The second scheme is of order three,
and the first and the third operator $\varphi$-functions have to be evaluated at each time step.
As in Experiment \ref{exa2}, the reference solutions are obtained by \textsf{ode45} with an absolute tolerance of $10^{-20}$
and a relative tolerance of $2.22045\cdot10^{-14}$ by solving the vector-valued ODEs generated by DREs (\ref{5.2}).
To provide a comparative baseline we also include two matrix-based BDF methods \cite{Dieci}, denoted \textsf{BDF1} and \textsf{BDF2}, where the number denotes the order of the method.
In our experiments we use the MATLAB solver \textsf{care} from the control systems toolbox to solve the algebra Riccati equations (\ref{5.2})
appearing in the BDF schemes.

In Fig. \ref{fig5.3}, left we present accuracy plots for \textsf{exprb2}, \textsf{exrb3}, \textsf{BDF1} and \textsf{BDF2} for the system
over the integration interval $[0, 0.05]$ with the variable grid sizes $n=2^k$ for $k=4,5,\ldots,9.$
The vertical axis shows the relative error at the transient state $t =0.05$ and the horizontal axis gives the CPU time.
We can see that \textsf{exprb2} and \textsf{exrb3} are more accurate than BDF methods under the same time step size.
In Fig. \ref{fig5.3}, right we show the relative error against the computation time,
which demonstrates that \textsf{exprb2} and \textsf{exrb3} are more efficient
than BDF methods.

In Table \ref{tab5.3} we list the relative errors as well as the corresponding time (in seconds) obtained with
each method at the stable state $t=0.1$ with the grid size
$n=100$. It can be seen that \textsf{exprb2} and \textsf{exrb3} are more accuracy and cost less CPU time.
\end{example}

\begin{figure}[!htbp]
\begin{minipage}{0.5\linewidth}
\centering
\includegraphics[width=7cm,height=5cm]{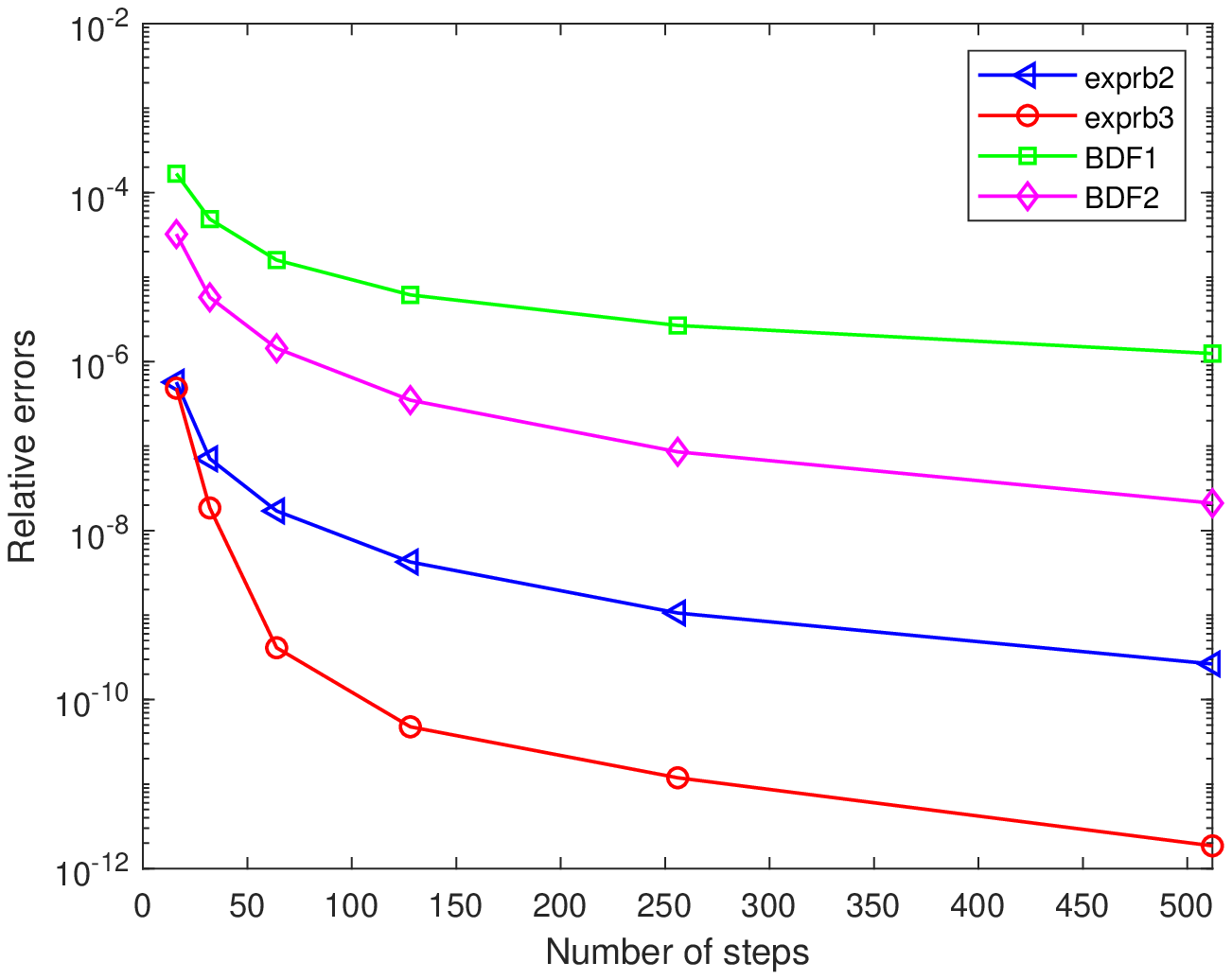}\\
\footnotesize{a. ~Accuracy plot}
\end{minipage}
\mbox{\hspace{-1.5cm}}
\begin{minipage}{0.5\linewidth}
\centering
\includegraphics[width=7cm,height=5cm]{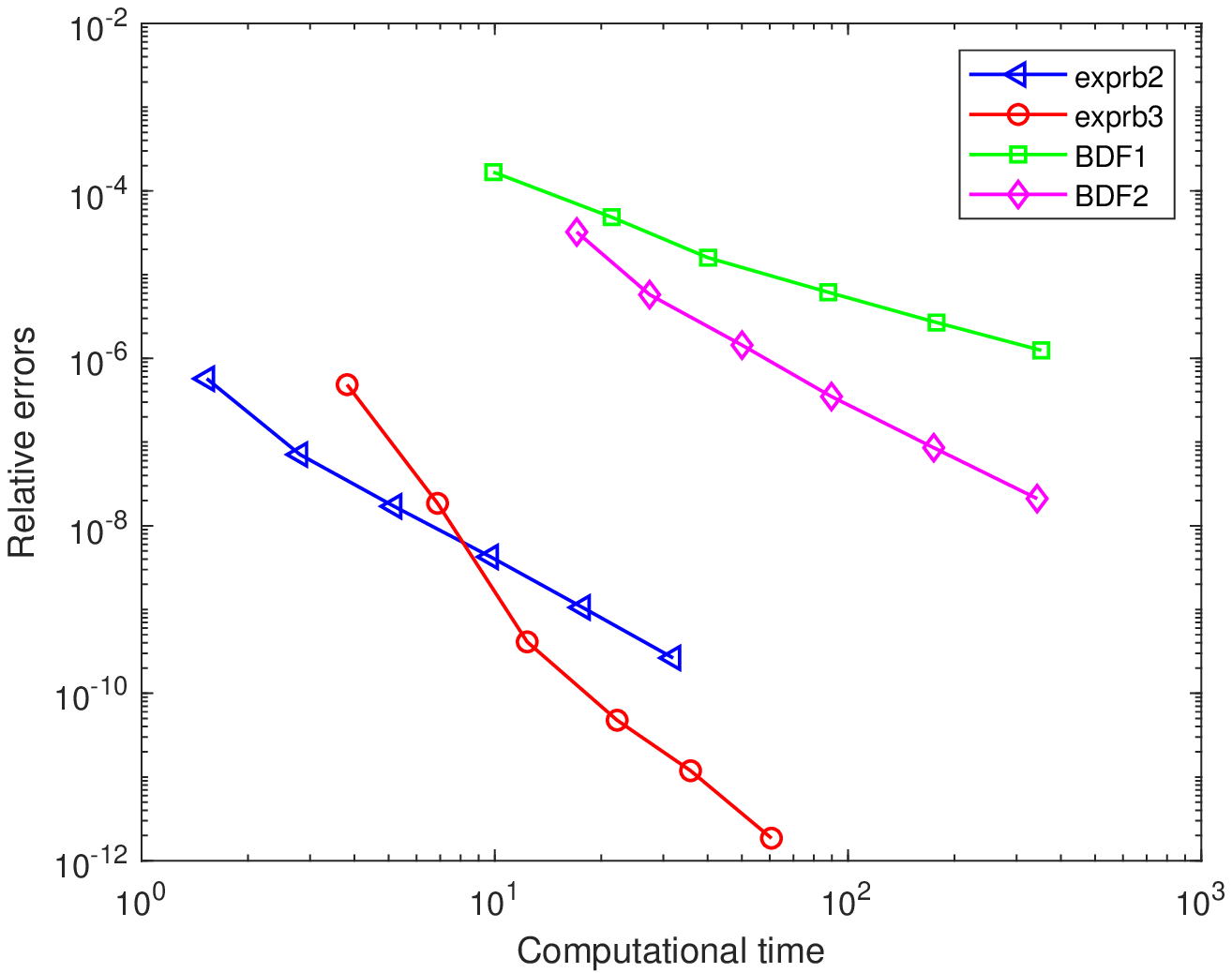}\\
\footnotesize{b.~Efficiency plot}
\end{minipage}\\
\caption{Results for the DREs for Experiment \ref{exa3}. Left: The relative errors versus the variable number of time steps
$k=2^{4},$ $(k=4,\cdots,9)$ at $t=0.05$. Right: The relative errors versus the computation time for the same problem.}\label{fig5.3}
\end{figure}

\begin{table}[H]
\setlength{\abovecaptionskip}{0.cm}
\setlength{\belowcaptionskip}{-0.3cm}
\caption{The CPU time (in seconds) and the relative errors when integrating DREs (\ref{5.2}) on [0, 0.1] of Experiment \ref{exa3}.}
 \label{tab5.3}
\begin{center}
\begin{tabular*}{\textwidth}{@{\extracolsep{\fill}}@{~~}lr|lr|lr|lr}
\toprule%
\multicolumn{2}{c|}{\textsf{expbr2}}&\multicolumn{2}{c|}{\textsf{expbr3}}&\multicolumn{2}{c|}{\textsf{BDF1}}&\multicolumn{1}{c}{\textsf{BDF2}}\\
\cline{1-8}
 error &time&error&time &error&time &error&time\\
\midrule%
4.6827e-14&8.11    &3.4002e-15&18.75   &2.5131e-10&67.82 &7.6810e-12&71.21 \\
\bottomrule
\end{tabular*}
\end{center}
\end{table}

\section{Conclusion}
The modified scaling and squaring method has been extended from the matrix $\varphi$-functions to the Lyapunov operator $\varphi$-functions.
Such operator functions constitute the building blocks of matrix-valued exponential integrators.
We have determined the key values of the order $m$ of the
Taylor approximation and the scaling parameter $s$ using a quasi-backward error analysis.
Numerical experiments illustrate that the method is efficient and reliable and can be used as a kernel for evaluating
the operator $\varphi$-functions in matrix-valued exponential integrators. We are currently
investigating the application of matrix-valued exponential integrators which use the method described in this paper
for solving the reduced LDEs and DREs by Krylov subspace methods.
In the future we also hope to develop low-rank approximations to large-scale Lyapunov operator
$\varphi$-functions and further devise efficient low-rank exponential integration schemes to solve large-scale MDEs.

\section*{Acknowledgements}
 This work was supported in part by the Jilin Scientific and Technological Development Program (Grant No. 20200201276JC)
 and the Natural Science Foundation of Jilin Province (Grant No. 20200822KJ),
 and the Natural Science Foundation of Changchun Normal University (Grant No. 002006059).


\begin{thebibliography}{99}

\bibitem{Abou} H. Abou-Kandil, G. Freiling, V. Ionescu and G. Jank, Matrix Riccati Equations in Control and Systems Theory, Birkh\"{a}user, Basel, Switzerland, 2003.
\bibitem{AlMohy2009}A. Al-Mohy and N. Higham, A new scaling and modified squaring algorithm for matrix functions, SIAM J. Matrix Anal. Appl., 31 (2009), pp. 970-989.
\bibitem{AlMohy2011}A. Al-Mohy and N. Higham, Computing the action of the matrix exponential, with an application to exponential integrators, SIAM J. Sci. Comput., 33 (2011), pp. 488-511.
\bibitem{Hached}V. Angelova, M. Hached and K. Jbilou, Approximate solution to large nonsymmetric differential Riccati problems with applications to transport theory, Numer. Linear Algebra Appl., 27 (2020), pp. 371--389.
\bibitem{Antoulas}A. C. Antoulas, Approximation of large-scale dynamical Systems, SIAM, Philadelphia, 2009.
\bibitem{Ascher} U.M. Ascher, R.M. Mattheij and R.G. Russell, Numerical solution of boundary value problems for ordinary differential equations, Prentice-Hall, Englewood Cliffs, NJ, 1988.

\bibitem{Behr} M. Behr, P. Benner and J. Heiland, Solution Formulas for Differential Sylvester and
Lyapunov Equations, Calcolo, 56 (4) (2019), pp. 1-33.

\bibitem{Benner01}P. Benner and H. Mena, Rosenbrock methods for solving Riccati differential equations, IEEE Trans. Autom. Control, 58 (2013), pp. 2950-2956.
\bibitem{Berland07} H. Berland and B. Skaflestad and W.M. Wright, EXPINT---a MATLAB Package for Exponential Integrators, ACM Trans. Math. Software, 33 (1) (2007), Article 4.

\bibitem{Caliari19}M. Caliari and F. Zivcovich, On-the-fly backward error estimate for matrix exponential approximation by Taylor algorithm, J. Comput. Appl. Math.,
346 (2019), pp. 532-548.
\bibitem{Choi}C.H. Choi and A. J. Laub, Efficient matrix-valued algorithms for solving stiff Riccati differential equations,  IEEE Trans. Autom. Control, 35 (1990), pp. 770-776.
\bibitem{Higham2003}I. Davies and N.J. Higham, A Schur-Parlett algorithm for computing matrix functions,
SIAM J. Matrix Anal. Appl., 25 (2003), pp. 464-485.
\bibitem{Defez2018}E. Defez and J. Ib\'{a}\~{n}ez, J. Sastre, J. Peinado and P. Alonso,
A new efficient and accurate spline algorithm for the matrix exponential computation, J. Comput. Appl. Math., 337 (2018), pp. 354-365.
\bibitem{Dieci}L. Dieci, Numerical integration of the differential Riccati equation and some
related issues, SIAM J. Numer. Anal., 29 (1992), pp. 781-815.
\bibitem{DP00}L. Dieci and A. Papini, Pad\'{e} approximation for the exponential of a block triangular matrix,
Linear Algebra Appl., 308 (2000), 183-202.
\bibitem{DM02} E.D. Dolan and J.J. Mor\'{e}, Benchmarking optimization software with performance profiles, Math. Program, 91 (2002), pp. 201-213.
\bibitem{Tokman18}S. Gaudreault, G. Rainwater, and M. Tokman, KIOPS: A fast adaptive
Krylov subspace solver for exponential integrators, J. Comput. Phys., 372 (1) (2018), pp. 236-255.

\bibitem{Higham2005}N.J. Higham, The scaling and squaring method for the matrix exponential revisited, SIAM J.
Matrix Anal. Appl., 26 (2005), pp. 1179-1193.
\bibitem{Higham}N.J. Higham, Functions of matrices: theory and computation, SIAM, Philadelphia, 2008.
\bibitem{Highamtool}N.J. Higham, The Matrix Computation Toolbox, \url{http://www.ma.man.ac.uk/~higham/mctoolbox}.
\bibitem{Higham00}N.J. Higham and F. Tisseur, A block algorithm for matrix 1-norm estimation, with an
application to 1-norm pseudospectra, SIAM J. Matrix Anal. Appl., 21 (2000), pp. 1185-1201.
\bibitem{MH2} M. Hochbruck and A. Ostermann, Explicit Exponential Runge-Kutta Methods for Semilinear Parabolic Problems, SIAM J. Numer. Anal., 43 (2006), pp. 1069-1090.
\bibitem{Hochbruck2010}M. Hochbruck and A. Ostermann, Exponential Integrators, Acta Numer., 19 (2010), pp. 209-286.

\bibitem{Jacobs} O.L.R. Jacobs, Introduction to Control Theory, Oxford Science Publications, Oxford, UK, 2nd ed., 1993.

\bibitem{KL1998}C.S. Kenney and A.J. Laub, A Schur-Fr\'{e}chet algorithm for computing the logarithm and
exponential of a matrix, SIAM J. Matrix Anal. Appl., 19 (1998), pp. 640-663.
\bibitem{Simoncini20} G. Kirsten and  V. Simoncini, Order reduction methods for solving large-scale differential matrix Riccati equations,
SIAM J. Sci. Comput., 42 (4) (2020), pp. 2182-2205.
\bibitem{Koskela}A. Koskela and H. Mena, A structure preserving Krylov subspace method for large scale differential Riccati equations, 2017, arXiv preprint,
\url{arXiv:1705.07507v1}.
\bibitem{Kucera} V. Ku\v{c}era, A review of the matrix Riccati equation, Kybernetika, 9 (1973), pp. 42-61.

\bibitem{Li2021} D.P. Li, X.Y. Zhang and R.Y. Liu, Exponential integrators for large-scale stiff Riccati differential equation, J. comput. Appl. Math., 389 (2021), 113360.
\bibitem{Li2022}D.P. Li, S.Y. Yang and J.M. Lan, Efficient and accurate computation for the $\varphi$-functions arising from exponential integrators, Calcolo, 59 (1) 2022, pp. 1-24.
\bibitem{Luan2013} V.T. Luan and A. Ostermann, Exponential B-series: the stiff case, SIAM J. Numer. Anal., 51 (2013), pp. 3431-3445.

\bibitem{Mena} H. Mena, A. Ostermann, L. Pfurtscheller and C. Piazzola, Numerical low-rank approximation of matrix
differential equations, J. comput. Appl. Math., 340 (2018), 602-614.
\bibitem{BV2005}B.V. Minchev and W.M. Wright, A review of exponential integrators for first order semi-linear problems, Tech. report 2/05, Department of Mathematics, NTNU, 2005.
\bibitem{Moler2003} C. Moler and C.V. Loan, Nineteen dubious ways to compute the exponential of a matrix, twenty-five years
later, SIAM Review, 45 (2003), pp. 3-49.

\bibitem{NH1995}I. Najfeld and T.F. Havel, Derivatives of the matrix exponential and their computation,
Adv. in Appl. Math., 16 (1995), pp. 321-375.
\bibitem{Niesen2012}J. Niesen and W. Wright, Algorithm 919: A Krylov subspace algorithm for evaluating the phi-functions
appearing in exponential integrators, ACM Trans. Math. Softw., 38(3) (2012), Article 22.

\bibitem{Paterson}M.S. Paterson and L.J. Stockmeyer, On the number of nonscalar multiplications necessary to evaluate polynomials, SIAM J. Comput., 2 (1) (1973), pp. 60-66.
\bibitem{Penzl} T. Penzl, LYAPACK: A MATLAB Toolbox for Large Lyapunov and Riccati Equations, Model Reduction Problems, and Linear-Quadratic Optimal Control Problems, Users' Guide (Version 1.0), 1999.

\bibitem{Sastre19}J. Sastre, J. Ib\'{a}\~{n}ez and E. Defez, Boosting the computation of the matrix
exponential, Appl. Math. Comput., 340 (2019), pp. 206-220.

\bibitem{Sastre2015}J. Sastre, J. Ib\'{a}\~{n}ez, E. Defez and P. Ruiz, New Scaling-Squaring Taylor Algorithms for Computing the Matrix Exponential,
 SIAM J. Sci. Comput., 37 (1) (2015), pp. 439-455.
\bibitem{Sidje1998} R.B. Sidje, Expokit: A software package for computing matrix exponentials, ACM Trans. Math. Softw., 24 (1998), pp. 130-156.
\bibitem{Skaflestad}B. Skaflestad and W.M. Wright, The scaling and modified squaring method for matrix functions
related to the exponential, Appl. Numer. Math., 59 (2009), pp. 783-799.

\bibitem{Stillfjord2}T. Stillfjord, Adaptive high-order splitting schemes for large-scale
differential Riccati equations, Numer. Algor., 78 (2018), pp. 1129-1151.
\bibitem{Suhov}A. Y. Suhov, An accurate polynomial approximation of exponential integrators, J. Sci. Comput., 60 (2014), pp. 684-698.

\bibitem{Ward} R.C., Ward, Numerical computation of the matrix exponential with accuracy estimate, SIAM J. Numer. Anal., 14 (1977), pp. 600-610.
\end{thebibliography}
\end{document}